\documentclass[a4paper,11pt]{amsart}

\usepackage{amssymb,amsmath,amsthm,mathrsfs,enumerate,graphicx, color}
\usepackage[pdfpagelabels,colorlinks,linkcolor=blue,citecolor=black,urlcolor=blue]{hyperref}
\usepackage{esint}
\usepackage{tikz}
\usetikzlibrary{arrows}
\usepackage{citeref}

\newtheorem{thm}{Theorem}[section]

\newtheorem{lem}[thm]{Lemma}
\newtheorem{prop}[thm]{Proposition}

\newcommand{\txg}{\Big(1+ \f{\sqrt t}{x}\Big)^{\gamma_\nu}}
\newcommand{\tyg}{\Big(1+ \f{\sqrt t}{y}\Big)^{\gamma_\nu}}

\newcommand{\lesi}{\lesssim}

\newcommand{\f}{\frac}

\newcommand{\su}{\subset}
\newcommand{\vc}{\infty}

\title[Some singular integrals in the Laguerre setting]{Weighted norm inequalities of some singular integrals  associated with  Laguerre expansions}         

\author[T. A. Bui]{The Anh Bui}
\address{School of Mathematical and Physical Sciences, Macquarie University, NSW 2109,
	Australia}
\email{the.bui@mq.edu.au}


\begin{document}

\begin{abstract}
Let $\nu=(\nu_1,\ldots,\nu_n)\in (-1,\vc)^n$, $n\ge 1$, and let $\mathcal{L}_\nu$ be a self-adjoint extension of the differential operator 
\[
L_\nu := \sum_{i=1}^n \left[-\frac{\partial^2}{\partial x_i^2} + x_i^2 + \frac{1}{x_i^2}(\nu_i^2 - \frac{1}{4})\right]
\] 
on $C_c^\infty(\mathbb{R}_+^n)$ as the natural domain. In this paper, we investigate the weighted estimates of singular integrals in the Laguerre setting including the maximal function, the Riesz transform and the square functions associated to the Laguerre operator $\mathcal L_\nu$. In the special case of the Riesz transform, the paper completes the description of the Riesz transform for the full range of $\nu\in (-1,\vc)^n$ which significant improves the result in [J.  Funct. Anal.  244 (2007),  399--443] for $\nu_i\ge -1/2, \nu_i\notin (-1/2,1/2)$ for $i=1,\ldots, n$.

\end{abstract}
\date{}

\maketitle

\tableofcontents

\section{Introduction}\label{sec: intro}
We now recall some backgrounds on the Laguerre operators in \cite{NS}. For each $\nu =(\nu_1,\ldots, \nu_n)\in (-1,\vc)^n$, we consider the Laguerre differential operator 
\[
\begin{aligned}
	L_{\nu}
	&=\sum_{i=1}^n\Big[-\f{\partial^2 }{\partial x_i^2} + x_i^2+\frac{1}{x_i^2}\Big(\nu_i^2 - \frac{1}{4}\Big)\Big], \quad \quad x \in (0,\vc)^n.
\end{aligned}
\]
The $j$-th partial derivative associated with $L_{\nu}$ is given by
\[
\delta_j = \frac{\partial}{\partial x_j} + x_j-\frac{1}{x_j}\Big(\nu_j + \f{1}{2}\Big).
\]
Then the  adjoint of $\delta_j$ in $L^2(\mathbb{R}^n_+)$ is
\[
\delta_j^* = -\frac{\partial}{\partial x_j} + x_j-\frac{1}{x_j}\Big(\nu_j + \f{1}{2}\Big).
\]
It is straightforward that
\[
\sum_{i=1}^{n} \delta_i^* \delta_i = L_{\nu} {-2}(|\nu| + n).
\]
Let $k = (k_1, \ldots, k_n) \in \mathbb{N}^n$, $\mathbb N = \{0, 1, \ldots\}$, and $\nu = (\nu_1, \ldots, \nu_n) \in (-1, \infty)^n$ be multi-indices. The Laguerre function $\varphi_k^{\nu}$ on $\mathbb{R}^n_+$ is defined as
\[
\varphi_k^{\nu}(x) = \varphi^{\nu_1}_{k_1}(x_1) \ldots \varphi^{\nu_n}_{k_n}(x_n), \quad x = (x_1, \ldots, x_n) \in \mathbb{R}^n_+,
\]
where $\varphi^{\nu_i}_{k_i}$ are the one-dimensional Laguerre functions
\[
\varphi^{\nu_i}_{k_i}(x_i) = \Big(\f{2\Gamma(k_i+1)}{\Gamma(k_i+\nu_i+1)}\Big)^{1/2}L_{\nu_i}^{k_i}(x_i^2)x_i^{\nu_i + 1/2}e^{-x_i^2/2}, \quad x_i > 0, \quad i = 1, \ldots, n,
\]
given $\nu > -1$ and $k \in \mathbb{N}$, $L_{\nu}^k$ denotes the Laguerre polynomial of degree $k$ and order $\nu$ outlined in  \cite[p.76]{L}.

Then it is well-known that the system $\{\varphi_k^{\nu} : k \in \mathbb{N}^n\}$ is an orthonormal basis of $L^2(\mathbb{R}^n_+, dx)$. Moreover, Each $\varphi_k^{\nu}$ is an eigenfunction of  $L_{\nu}$ corresponding to the eigenvalue of $4|k| + 2|\nu| + 2n$, i.e.,
\[
L_{\nu}\varphi_k^{\nu} = (4|k| + 2|\nu| + 2n) \varphi_k^{\nu},
\]
where $|\nu| = \nu_1 + \ldots + \nu_n$ and  $|k| = k_1 + \ldots + k_n$. The operator $L_{\nu}$ is positive and symmetric in $L^2(\mathbb{R}^n_+, dx)$.

If $\nu = \left(-\frac{1}{2}, \ldots, -\frac{1}{2}\right)$, then $L_{\nu}$ becomes the harmonic oscillator $-\Delta + |x|^2$ on $\mathbb R^n_+$. Note that the Riesz transforms related to the harmonic oscillator $-\Delta + |x|^2$ on $\mathbb R^n$ were investigated in \cite{ST, Th}.

The operator
\[
\mathcal{L}_{\nu}f = \sum_{k\in \mathbb{N}^{n}} (4|k| + 2|\nu| + 2n) \langle f, \varphi_{k}^\nu\rangle \varphi_{k}^\nu 
\]
defined on the domain $\text{Dom}\, \mathcal{L}_{\nu} = \{f \in L^2(\mathbb{R}^n_+): \sum_{k\in \mathbb{N}^{d}} (4|k| + 2|\nu| + 2n) |\langle f, \varphi_{k}^\nu\rangle|^2 < \infty\}$ is a self-adjoint extension of $L_{\nu}$ (the inclusion $C^{\infty}_c(\mathbb{R}^{d}_{+}) \subset \text{Dom}\, \mathcal{L}_{\nu}$ may be easily verified), has the discrete spectrum $\{4\ell + 2|\nu| + 2n: \ell \in \mathbb{N}\}$, and admits the spectral decomposition
\[
\mathcal{L}_{\nu}f = \sum_{\ell=0}^{\infty} (4\ell + 2|\nu| + 2n) P_{\nu,\ell}f,
\]
where the spectral projections are
\[
P_{\nu,\ell}f = \sum_{|k|=\ell} \langle f, \varphi_{k}^\nu\rangle \varphi_{k}^\nu.
\]

Moreover,
\begin{equation}\label{eq- delta and eigenvector}
	\delta_j \varphi_k^\nu =-2\sqrt{k_j} \varphi_{k-e_j}^{\nu+e_j}, \ \ \delta_j^* \varphi_k^\nu =-2\sqrt{k_j+1} \varphi_{k+e_j}^{\nu-e_j},
\end{equation}
where $\{e_1,\ldots, e_n\}$ is the standard basis for $\mathbb R^n$. Here and later on we use the convention that $\varphi_{k-e_j}^{\nu+e_j}=0$ if $k_j-1<0$ and $\varphi_{k+e_j}^{\nu-e_j}=0$ if $\nu_j-1<0$. See for example \cite{NS}.

Muckenhoupt and Stein \cite{MS} initiated the investigation into the boundedness of Riesz transforms within the context of orthogonal expansions. The harmonic analysis associated with orthogonal expansions has since attracted the attention of numerous authors, as evidenced in works such as \cite{Betancor, Betancor2, Muc, NS, NS2, ST, Th} and the references therein. Motivated by these contributions in the Laguerre setting, this paper aims to establish weighted norm inequalities for singular integrals, including the maximal function, the Riesz transform, and the square functions associated with the Laguerre operator $\mathcal{L}_\nu$. Before delving into the details, we introduce some notations. For  $\nu\in (-1,\vc)^n$, we set
$$
\mathcal I_\nu=\{j: -1<\nu_j<-1/2\}
$$
and then define
\begin{equation}\label{eq-gamma nu}
	\gamma_\nu = \begin{cases}
		\max\{-1/2-\nu_j: j\in \mathcal I_\nu\}, \ \ & \ \text{if} \  \mathcal I_\nu\ne \emptyset, \\
		0, \ \ &\ \text{if} \  \mathcal I_\nu= \emptyset.
	\end{cases}
\end{equation}
In the 1-dimensional case when $\nu \in (-1,\vc)$, we have 
 \[
 	\gamma_\nu = \begin{cases}
 	-1/2-\nu, \ \ &  -1<\nu<-1/2, \\
 	0, \ \ &\ \nu\ge -1/2.
 \end{cases}
 \]

We first consider the maximal function defined by
\[
\mathcal M_{\mathcal L_\nu} f =\sup_{t>0} |e^{-t\mathcal L_\nu}f|.
\]
For the 1-dimensional case when $n=1$, the weighted $L^p$-estimates (with power weights) were investigated in \cite[Theorem 4.1]{Betancor2} using the transference method. It was proved in \cite[Theorem 4.1]{Betancor2} that the maximal function $\mathcal{M}_{\mathcal{L}\nu}$ is bounded on $L^p_{x^\sigma}(\mathbb{R}_+)$ whenever $-1-p(\nu+3/2)<\sigma<p(\nu+3/2)-1$. This method does not seem to imply the weighted inequality with Muckenhoupt weights. Hence, our first main result surpasses this limitation to prove the Muckenhoupt weighted $L^p$ inequality for the maximal function. In what follows, for the $A_p$ and $RH_q$ classes of weights, we refer to Section 2.1 for the precise definitions.

\begin{thm}\label{main thm 0}
	Let $\nu \in (-1,\vc)$. Then the maximal function $\mathcal M_{\mathcal L_\nu}$ is bounded on $L^p_w(\mathbb R_+)$ for all $(\f{1}{\gamma_\nu})'<p<\f{1}{\gamma_\nu}$ and $w\in A_{p(1-\gamma_\nu)}(\mathbb R_+)\cap RH_{(\f{1}{p\gamma_\nu})'}(\mathbb R_+)$ with the convention $\f{1}{0}=\vc$.
\end{thm}
In the case $\nu\ge -1/2$, Theorem \ref{main thm 0} states that the maximal function $\mathcal M_{\mathcal L_\nu}$ is bounded on $L^p_w(\mathbb R_+)$ for all $1<p<\vc$ and $w\in A_p$. For  $-1<\nu<-1/2$ and $w(x)=x^\sigma$, Theorem \ref{main thm 0} asserts that the maximal function $\mathcal{M}_{\mathcal{L}\nu}$ is bounded on $L^p_{x^\sigma}(\mathbb{R}_+)$ whenever $1<p<\infty$ and $-1-p(\nu+1/2)<\sigma<p(\nu+3/2)-1$, which recovers the result in \cite[Theorem 4.1]{Betancor2}. We are currently uncertain whether the estimate in Theorem \ref{main thm 0} holds true in the higher-dimensional case of $\mathbb{R}^n_+$, where $n\geq 2$.

\bigskip

We next concern the Riesz transform $R_\nu =(R_\nu^1,\ldots, R_\nu^n)$ defined by
\[
R_\nu^j = \delta_j \mathcal L_\nu^{-1/2}, \ \ j=1,\ldots, n.
\]
It is established in \cite{NS} that the Riesz transform is bounded on $L^2(\mathbb{R}^n_+)$. In \cite[Theorem 3.3]{NS}, it was proved that if $\nu_i\geq -\frac{1}{2}$, $\nu_i \notin \left(-\frac{1}{2}, \frac{1}{2}\right)$ for $i = 1, \ldots, n$, then the Riesz transform $R_\nu$ is a Calderón--Zygmund operator, and hence, it is bounded on $L^p_w(\mathbb{R}^n_+)$ for $1<p<\infty$ and $w\in A_p(\mathbb{R}^n_+)$. This result has been recently improved by the author in \cite{B}, which demonstrates that the Riesz transform $R_\nu$ is a Calderón--Zygmund operator if $\nu\in [-1/2,\infty)^n$. The case $\nu\in (-1,\infty)^n\backslash [-1/2,\infty)^n$ remains open and is only known for $n=1$ with power weights estimates, as shown in \cite{Betancor, Betancor2}. The following result not only fills the gap when $\nu \in (-1,\infty)^n\backslash [-1/2,\infty)^n$ but also extends the weighted estimates to Muckenhoupt weights. More precisely, denote  $\delta=(\delta_1,\ldots, \delta_n)$ and $\delta^*=(\delta^*_1,\ldots, \delta^*_n)$. Then, we have:

\begin{thm}
	\label{main thm 1}
	Let $\nu\in (-1,\vc)^n$ and $j=1,\ldots, n$. Then for $\f{1}{1-\gamma_\nu}<p< \f{1}{\gamma_{\nu+e_j}}$ with the convention $\f{1}{0}=\vc$, the Riesz transform $R^j_\nu:=\delta _j\mathcal L_\nu^{-1/2}$ is bounded on $L^p_w(\mathbb R^n)$ with $w\in A_{(1-\gamma_\nu)p}\cap RH_{(\f{1}{p\gamma_{\nu+e_j}})'}$.
\end{thm}

Note that Theorem \ref{main thm 1} is new, even for the unweighted estimate when $n\geq 2$. In the specific case when $n=1$, Theorem \ref{main thm 1} reads that the Riesz transforms $\delta \mathcal L_\nu^{-1/2}$ is bounded on $L^p_w(\mathbb R_+)$ for $\f{1}{1-\gamma_\nu}<p<\vc$ and $w\in A_{p(1-\gamma_\nu)}$, where $\gamma_\nu =\max\{-\nu-1/2, 0\}$. This marks the first instance where weighted norm inequalities for the Riesz transform have been investigated with Muckenhoupt weights. 

Finally, as a by product, we consider the following square functions
	\begin{equation}
	\label{eq-square function}
	S^j_{\mathcal L_\nu}f = \Big(\int_0^\vc |\sqrt t \delta_j e^{-t\mathcal L_\nu}f|^2\f{dt}{t} \Big)^{1/2}, \ \ j=1,\ldots,n,
\end{equation}
and
	\begin{equation}
	\label{eq-square function tLetL}
	G_{\mathcal L_\nu}f = \Big(\int_0^\vc |t\mathcal L_\nu  e^{-t\mathcal L_\nu}f|^2\f{dt}{t} \Big)^{1/2}.
\end{equation}
Then we also obtained the boundedness of these square functions on the weighted $L^p(\mathbb R^n_+)$.
\begin{thm}
	\label{main thm 2}
	Let $\nu\in (-1,\vc)^n$ and $j=1,\ldots, n$. Then we have
	\begin{enumerate}
		\item[(a)] the square functions $G_{\mathcal L_\nu}$ is bounded on $L^p_w(\mathbb R^n)$ for $\f{1}{1-\gamma_\nu}<p<\f{1}{\gamma_{\nu}}$ and $w\in A_{(1-\gamma_\nu)p}\cap RH_{(\f{1}{p\gamma_{\nu}})'}$  with the convention $\f{1}{0}=\infty$.
		
		\item[(b)] the square functions $S^j_{\mathcal L_\nu}$ is bounded on $L^p_w(\mathbb R^n)$ for $\f{1}{1-\gamma_\nu}<p<\f{1}{\gamma_{\nu+e_j}}$ and $w\in A_{(1-\gamma_\nu)p}\cap RH_{(\f{1}{p\gamma_{\nu+e_j}})'}$ with the convention $\f{1}{0}=\infty$.
	\end{enumerate}
	
\end{thm}

Finally, needless to say that our approach also implies the boundedness of the area integral considered in \cite{Betancor3}. However, due to the length of the paper, we will not present the results here and would like to leave it for interested readers.

\bigskip

The paper is organized as follows. In Section 2, we recall the definitions of Muckenhoupt weights and two theorems regarding the boundedness criteria of singular integrals beyond the Calderón-Zygmund theory. Section 3 is dedicated to proving kernel estimates related to the heat semigroup $e^{-t\mathcal{L}_\nu}$. Section 4 is devoted to the proofs of the main results.

\bigskip
 
 Throughout the paper, we always use $C$ and $c$ to denote positive constants that are independent of the main parameters involved but whose values may differ from line to line. We will write $A\lesi B$ if there is a universal constant $C$ so that $A\leq CB$ and $A\sim B$ if $A\lesi B$ and $B\lesi A$. For $a \in \mathbb{R}$, we denote the integer part of $a$ by $\lfloor a\rfloor$.  For a given ball $B$, unless specified otherwise, we shall use $x_B$ to denote the center and $r_B$ for the radius of the ball.
 
In the whole paper, we will often use the following inequality without any explanation $e^{-x}\le c(\alpha) x^{-\alpha}$ for any $\alpha>0$ and $x>0$.
\section{Preliminaries}
\subsection{Muckenhoupt weights and the Hardy-Littlewood maximal function}
We start with some notations which will be used frequently. For a measurable subset $E\subset \mathbb R^n_+$ and a measurable function $f$ we denote
$$
\fint_E f(x)dx=\f{1}{|E|}\int_E f(x)dx.
$$
Given a ball $B$, we denote $S_j(B)=2^{j}B\backslash 2^{j-1}B$ for $j=1, 2, 3, \ldots$, and we set $S_0(B)=B$.

Let $1\leq q<\infty$. A nonnegative locally integrable function $w$ belongs to the  Muckenhoupt class  $A_q$, say $w\in A_q$, if there exists a positive constant $C$ so that
$$\Big(\fint_B w(x)dx\Big)\Big(\fint_B w^{-1/(q-1)}(x)dx\Big)^{q-1}\leq C, \quad\mbox{if}\; 1<q<\infty$$
and
$$
\fint_B w(x)dx\leq C \mathop{\mbox{ess-inf}}\limits_{x\in B}w(x),\quad{\rm if}\; q=1,
$$
for all balls $B$ in $\mathbb R^d$. We define $A_\infty =\bigcup_{1\le p<\vc} A_p$.


The reverse H\"older classes are defined in the following way: $w\in RH_r, 1 < r < \infty$, if there is a constant $C$ such that for
any ball $B \subset \mathbb R^n_+$,
$$
\Big(\fint_B w^r (x) dx\Big)^{1/r} \leq C \fint_B w(x)dx.
$$
The endpoint $r = \infty$ is given by the condition: $w \in RH_\infty$ whenever, there is a constant $C$ such that for any ball
$B \subset \mathbb R^n_+$,
$$
w(x)\leq C \fint_B w(y)dy  \ \text{for a.e. $x\in B$}.
$$

For $w \in A_\vc$ and $0< p <\infty$, the weighted space $L^p_w(\mathbb R^n_+)$ is defined as  the space of $w(x)dx$-measurable functions $f$ such that
$$\|f\|_{L^p_w(\mathbb R^n_+)}:=\Big(\int_{\mathbb R^n_+} |f(x)|^p w(x)dx\Big)^{1/p}<\infty.$$

It is well-known that the power weight $w(x)=x^\alpha \in A_p$ if and only if $-n<\alpha<n(p-1)$. Moreover, $w(x)=x^\alpha \in RH_q$ if and only if $\alpha q>-n$.

We sum up some of the properties of Muckenhoupt classes and reverse H\"older classes in the following
results. See \cite{Du, JN}.
\begin{lem}\label{weightedlemma1}
	The following properties hold:
	\begin{enumerate}[{\rm (i)}]
		\item $w\in A_p, 1<p<\vc$ if and only if $w^{1-p'}\in A_{p'}$.
		\item $A_1\subset A_p\subset A_q$ for $1\leq p\leq q\leq \infty$.
		\item $ RH_q \su RH_p$ for $1< p\leq q< \infty$.
		\item If $w \in A_p, 1 < p < \vc$, then there exists $1 < q < p$ such that $w \in A_q$.
		\item If $w \in RH_q, 1 < q < \vc$, then there exists $q < p < \infty$ such that $w \in RH_p$.
		\item $A_\vc =\cup_{1\leq p<\vc}A_p = \cup_{1< p< \vc}RH_p$.
		\item Let $1<p_0 < p < q_0<\vc$. Then we have
		$$
		w\in A_{\f{p}{p_0}}\cap RH_{(\f{q_0}{p})'}\Longleftrightarrow
		w^{1-p'}\in A_{\f{p'}{q'_0}}\cap RH_{(\f{p'_0}{p'})'}.
		$$
	\end{enumerate}
\end{lem}

For $r>0$, the Hardy-Littlewood maximal function $\mathcal{M}_r$ is defined by
$$
\mathcal{M}_rf(x)=\sup_{B\ni x}\Big(\f{1}{|B|}\int_B|f(y)|^r\,dy\Big)^{1/r}, \ x\in \mathbb R^n_+,
$$
where the supremum is taken over all balls $B$ containing $x$. When $r=1$, we write $\mathcal{M}$ instead of $\mathcal{M}_1$.

We now record the following results concerning the weak type estimates and the weighted estimates of the maximal functions.
\begin{lem}[\cite{S}]\label{Lem-maximalfunction}
	Let $0< r<\vc$. Then we have for $p>r$ and $w\in A_{p/r}$,
	$$
	\|\mathcal{M}_rf\|_{L^p_w}\lesi \|f\|_{L^p_w}.
	$$
\end{lem}

We will end this section by some simple estimates whose proofs will be omitted.

\begin{lem}\label{lem-elementary}
	Let $\epsilon>0$. Then we have:
	\begin{enumerate}[{\rm (i)}]
		\item for every $x\in \mathbb R^n_+$ and $f\in L^1_{\rm loc}(\mathbb R^n_+)$,
		$$\displaystyle \sup_{t>0}\int_{\mathbb R^n_+} \f{1}{t^n}\Big(1+\f{|x-y|}{t}\Big)^{n+\epsilon}f(y)dy \le C(\epsilon)\mathcal M f(x);$$
		
		\item  for every $x\in \mathbb R^n_+$,
		$$\displaystyle \sup_{t>0}\int_{\mathbb R^n_+} \f{1}{t^n}\Big(1+\f{|x-y|}{t}\Big)^{n+\epsilon}f(y)dy \le C(\epsilon);$$
		
		\item for every $t>0$,
		\[
		\int_{|x|<t} \Big(\f{t}{|x|}\Big)^{n-\epsilon} dx \lesi t^{-n}.
		\]
	\end{enumerate}
\end{lem}
\subsection{Singular integrals beyond the Calder\'on-Zygmund theory} In this section we recall some criteria for singular integrals beyond the Calder\'on-Zygmund theory to be bounded on the weighted Lebesgue spaces in \cite{AM, BZ}.

 An operator $T$ defined on $L^2(\mathbb R^n_+)$ is said to be a \emph{linearizable operator} if there exists a Banach space $\mathbb{B}$ and a linear operator $U$ from $L^2(\mathbb R^n_+)$ into $L^2(\mathbb R^n_+, \mathbb{B})$ so that
$$
|Tf(x)|=\|Uf(x)\|_{\mathbb{B}}
$$
for all $f\in L^2(\mathbb R^n_+)$ and a.e. $x\in \mathbb R^n_+$.

It can be verified that a linearizable operator is a sublinear operator. The class of linearizable operator includes linear operators, maximal operators and square functions.

We first recall a theorem which is taken from \cite[Theorem 6.6]{BZ} on a criterion for the singular integrals to be bounded on the weighted Lebesgue spaces.
\begin{thm}\label{BZ-thm}
	Let $1\leq p_0< q_0< \vc$ and let $T$ be a linearizable operator. Assume that $T$ can be extended to be bounded on $L^{q_0}$. Assume that there exists a family of operators $\{\mathcal{A}_t\}_{t>0}$ satisfying that for $j\geq 2$ and every ball $B$
	\begin{equation}\label{eq1-BZ}
		\Big(\fint_{S_j(B)}|T(I-\mathcal{A}_{r_B})f|^{q_0}dx\Big)^{1/q_0}\leq
		\alpha(j)\Big(\fint_B |f|^{q_0}dx\Big)^{1/q_0},
	\end{equation}
	and
	\begin{equation}\label{eq1-BZ-bis}
		\Big(\fint_{S_j(B)}|\mathcal{A}_{r_B}f|^{q_0}dx\Big)^{1/q_0}\leq
		\alpha(j)\Big(\fint_B |f|^{p_0}dx\Big)^{1/p_0},
	\end{equation}
	for all $f$ supported in $B$. If $\sum_j \alpha(j)2^{jd}<\vc$, then $T$ is bounded on $L^p_w(\mathbb R^n_+)$ for all $p\in (p_0,q_0)$ and $w\in
	A_{\f{p}{p_0}}\cap RH_{(\f{q_0}{p})'}$.
\end{thm}
Note that \cite[Theorem 6.6]{BZ} proves Theorem \ref{BZ-thm} for $q_0=2$, but their arguments also work well for any value $q_0\in (1,\vc)$.

The following theorem is a direct consequence of \cite[Theorem 3.7]{AM} which give a sufficient conditions for a singular integral to be bounded on Lebesgue spaces which plays an important role in the sequel.

\begin{thm}\label{Martell-thm}
	
	Let $1\leq  p_0< q_0\leq \infty.$  Let $T$ be a bounded sublinear
	operator on $L^{p_0}(\mathbb R^n_+)$. Assume that there exists  a family of operators $\{\mathcal{A}_t\}_{t>0}$  satisfying that
	\begin{eqnarray}\label{e1-Martell}
		\Big( \fint_{B} \big| T(I-\mathcal{A}_{r_B})f\big|^{p_0}dx\Big)^{1/p_0} \leq
		C \mathcal{M}_{p_0}(f)(x),
	\end{eqnarray}
	and
	\begin{eqnarray}\label{e2-Martell}
		\Big( \fint_{B} \big| T\mathcal{A}_{r_B}f\big|^{q_0}dx\Big)^{1/q_0} \leq
		C \mathcal{M}_{p_0}(|Tf|)(x),
	\end{eqnarray}
	\noindent for all balls
	$B$ with radius $r_B$, all $f \in C^{\infty}_c(\mathbb R^n_+) $ and all $x\in B$. Then $T$ is bounded on $L^p_w(\mathbb R^n_+)$ for all
	$p_0<p<q_0$ and $w\in A_{p/p_0}\cap RH_{(q_0/p)'}$.
\end{thm}

\section{Some   kernel estimates}

This section is devoted to establishing some kernel estimates related to the heat kernel of $\mathcal L_\nu$. These estimates play an essential role in proving our main results. We begin by providing an explicit formula for the heat kernel of $\mathcal L_\nu$.

Let $\nu \in (-1,\vc)^n$. For each $j=1,\ldots, n$, similarly to $\mathcal L_\nu$, denote by $\mathcal L_{\nu_j}$ the self-adjoint extension of the differential operator 
\[
L_{\nu_j} :=  -\frac{\partial^2}{\partial x_j^2} + x_j^2 + \frac{1}{x_j^2}(\nu_j^2 - \frac{1}{4})
\] 
on $C_c^\infty(\mathbb{R}_+)$ as the natural domain. It is easy to see that 
\[
\mathcal L_\nu =\sum_{j=1}^n \mathcal L_{\nu_j}.
\]

Let $p_t^\nu(x,y)$ be the kernel of $e^{-t\mathcal L_\nu}$ and let $p_t^{\nu_j}(x_j,y_j)$ be the kernel of $e^{-t\mathcal L_{\nu_j}}$ for each $j=1,\ldots, n$. Then we have
\begin{equation}\label{eq- prod ptnu}
	p_t^\nu(x,y)=\prod_{j=1}^n p_t^{\nu_j}(x_j,y_j).
\end{equation}
For $\nu_j\ge -1/2$, $j=1,\ldots, n,$, the kernel of $e^{-t\mathcal L_{\nu_j}}$ is given by
\begin{equation}
	\label{eq1-ptxy}
	p_t^{\nu_j}(x_j,y_j)=\f{2(rx_jy_j)^{1/2}}{1-r}\exp\Big(-\f{1}{2}\f{1+r}{1-r}(x_j^2+y_j^2)\Big)I_{\nu_j}\Big(\f{2r^{1/2}}{1-r}x_jy_j\Big),
\end{equation}
where $r=e^{-4t}$ and $I_\alpha$ is the usual Bessel funtions of an imaginary argument defined by
\[
I_\alpha(z)=\sum_{k=0}^\vc \f{\Big(\f{z}{2}\Big)^{\alpha+2k}}{k! \Gamma(\alpha+k+1)}, \ \ \ \ \alpha >-1.
\]
See for example \cite{Dziu, NS}.

Note that for each $j=1,\ldots, n$, we can rewrite the kernel $p_t^{\nu_j}(x_j,y_j)$ as follows
\begin{equation}
	\label{eq2-ptxy}
	\begin{aligned}
		p_t^{\nu_j}(x_j,y_j)=\f{2(rx_jy_j)^{1/2}}{1-r}&\exp\Big(-\f{1}{2}\f{1+r}{1-r}|x_j-y_j|^2\Big)\exp\Big(-\f{1-r^{1/2}}{1+r^{1/2}}x_jy_j\Big)\\
		&\times \exp\Big(-\f{2r^{1/2}}{1-r}x_jy_j\Big)I_{\nu_j}\Big(\f{2r^{1/2}}{1-r}x_jy_j\Big),
	\end{aligned}
\end{equation}
where $r=e^{-4t}$.

The following  properties of the Bessel function $I_\alpha$  with $\alpha>-1$ are well-known and are taken from \cite{L}:
\begin{equation}
	\label{eq1-Inu}
	I_\alpha(z)\sim z^\alpha, \ \ \ 0<z\le 1,
\end{equation}
\begin{equation}
	\label{eq2-Inu}
	I_\alpha(z)= \f{e^z}{\sqrt{2\pi z}}+S_\alpha(z),
\end{equation}
where
\begin{equation}
	\label{eq3-Inu}
	|S_\alpha(z)|\le  Ce^zz^{-3/2}, \ \ z\ge 1,
\end{equation}
\begin{equation}
	\label{eq4-Inu}
	\f{d}{dz}(z^{-\alpha}I_\alpha(z))=z^{-\alpha}I_{\alpha+1}(z).
\end{equation}

Let $\nu > -1$, we have
\[
I_\alpha(z)-I_{\alpha+1}(z)=[I_\alpha(z)-I_{\alpha+2}(z)] -[I_{\alpha+1}(z)-I_{\alpha+2}(z)].
\]
Applying \cite[(5.7.9)]{L}, 
\[
I_\alpha(z)-I_{\alpha+2}(z) =\f{2(\alpha+1)}{z}I_{\alpha+1}(z).
\]
Hence,
\[
|I_\alpha(z)-I_{\alpha+1}(z)|\le \f{2(\alpha+1)}{z}I_{\alpha+1}(z)+|I_{\alpha+1}(z)-I_{\alpha+2}(z)|.
\]
On the other hand, since $\alpha+1>-1/2$, we have
\[
	0< I_{\alpha+1}(z)-I_{\alpha+2}(z)<2(\alpha+2)\f{I_{\alpha+2}(z)}{z}<2(\alpha+2)\f{I_{\alpha+1}(z)}{z}, \ \ \ z>0.
\]
See for example \cite{Na}.

Consequently,
\begin{equation}
	\label{eq5-Inu}
	|I_\alpha(z)-I_{\alpha+1}(z)|<(4\alpha +6)\f{I_{\alpha+1}(z)}{z}, \ \ \ \alpha>-1, z>0.
\end{equation}

\subsection{The case $n=1$}	In this case, recall that the function $\gamma_\nu$ defined by \eqref{eq-gamma nu} becomes
\begin{equation*} 
	\gamma_\nu = \begin{cases}
		 -1/2-\nu, & \ -1<\nu<-1/2, \\
		0, \ \ &\ \nu\ge -1/2.
	\end{cases}
\end{equation*}

We first prove the following result.
\begin{prop}\label{prop-heat kernel}
	Let $\nu \ge -1$. The the following holds true.
	\begin{enumerate}[{\rm (i)}]
		\item If $\nu>-1/2$, then
		\begin{equation}
			\label{eq-ptxy nu > -1/2}
			\begin{aligned}
				p_t^\nu(x,y)\lesi \f{e^{-t/2}}{\sqrt t}\exp\Big(-\f{|x-y|^2}{ct}\Big)\Big(1+\f{\sqrt t}{\rho(x)} +\f{\sqrt t}{\rho(y)}\Big)^{-(\nu+1/2)}
			\end{aligned}
		\end{equation}
		for all $t>0$ and all $x,  y\in (0,\vc)$, where $\rho(x)=\min\{x,1/x\}$.
		
		\item If $\nu=-1/2$, then for any $N>0$
		\begin{equation}
			\label{eq-ptxy nu = -1/2}
			\begin{aligned}
				p_t^\nu(x,y)\lesi \f{e^{-t/2}}{\sqrt t}\exp\Big(-\f{|x-y|^2}{ct}\Big)\Big(1+\f{\sqrt t}{\rho(x)} +\f{\sqrt t}{\rho(y)}\Big)^{-N}
			\end{aligned}
		\end{equation}
		for all $t>0$ and all $x,  y\in (0,\vc)$, where $\displaystyle \rho(x)=\f{1}{1+x}$.
		
		\item If $-1<\nu<-1/2$, then
		\begin{equation}
			\label{eq-ptxy nu < -1/2}
			\begin{aligned}
				p_t^\nu(x,y)\lesi \f{e^{-t/2}}{\sqrt t}\exp\Big(-\f{|x-y|^2}{ct}\Big)\Big(1+\f{\sqrt t}{x}\Big)^{\gamma_\nu}\Big(1+\f{\sqrt t}{y}\Big)^{\gamma_\nu}
			\end{aligned}
		\end{equation}
		for all $t>0$ and all $x,  y\in (0,\vc)$.
	\end{enumerate}
\end{prop}
\begin{proof}
	
	The items (i) and (ii) were proved in \cite[Proposition 3.1]{B}. It remains to prove (iii). To do this, we consider two cases.

	\noindent\textbf{Case 1: $0<t\le 1$.} In this situation, we have $r \sim r^{1/2}\sim 1$ and $1-r\sim 1-r^{1/2}\sim t$ where $r=e^{-4t}$. We now consider two subcases: $xy<t$ and $xy>t$.
	
	\textbf{Subcase 1.1: $xy<t$.} By \eqref{eq1-ptxy} and \eqref{eq1-Inu} we have
	\begin{equation}
		\label{eq-est1-proof Prop1}
		\begin{aligned}
			p_t^\nu(x,y)&\lesi \f{1}{e^{t}}\exp\Big(-\f{x^2+y^2}{ct}\Big)\Big(\f{t}{xy}\Big)^{\gamma_\nu},
		\end{aligned}
	\end{equation}
which implies \eqref{eq-ptxy nu < -1/2}. 
	
	\textbf{Subcase 1.2: $xy\ge t$.} By \eqref{eq2-ptxy}, \eqref{eq2-Inu} and \eqref{eq3-Inu}  we have
	\begin{equation}
		\label{ptxy subcase 1.2}
		\begin{aligned}
			p_t^\nu(x,y)&\lesi \f{1}{\sqrt{t}}\exp\Big(-\f{|x-y|^2}{ct}\Big)\exp\Big(-c'txy\Big)\\
			&\lesi \f{1}{\sqrt{t}}\exp\Big(-\f{|x-y|^2}{ct}\Big),
		\end{aligned}
	\end{equation}
	which implies \eqref{eq-ptxy nu < -1/2}.
	
	\bigskip
	
	\noindent\textbf{Case 2: $t> 1$.} In this situation, we have $1-r\sim 1+r\sim 1+r^{1/2}\sim 1-r^{1/2}\sim 1$, where $r=e^{-4t}$. We also consider two subcases: $xy<t$ and $xy>t$.
	
	\textbf{Subcase 2.1: $xy<e^{2t}$.} By \eqref{eq1-ptxy}, \eqref{eq1-Inu} and the fact $1-2\gamma_\nu>0$  we have
	\begin{equation}\label{eq- heat kernel xy< e2t}
		\begin{aligned}
			p_t^\nu(x,y)&\lesi \f{1}{e^{t}}\exp\Big(-c(x^2+y^2)\Big)\Big(\f{1}{e^{-2t}xy}\Big)^{\gamma_\nu}\\
			&\lesi \f{1}{e^{t(1-2\gamma_\nu)}}\exp\Big(-c(x^2+y^2)\Big)\Big(\f{1}{xy}\Big)^{\gamma_\nu}\\
			&\lesi \f{1}{\sqrt t}\exp\Big(-c(x^2+y^2)\Big)\Big(\f{t}{xy}\Big)^{\gamma_\nu}\\
		\end{aligned}
	\end{equation}
	which implies \eqref{eq-ptxy nu < -1/2}.
	
	\textbf{Subcase 2.2: $xy\ge e^{2t}$.} By \eqref{eq2-ptxy}, \eqref{eq2-Inu} and \eqref{eq3-Inu}  we have
	\begin{equation}
		\label{ptxy subcase 2.2}
		\begin{aligned}
			p_t^\nu(x,y)&\lesi \f{1}{e^t}\exp\Big(-c|x-y|^2\Big)\exp\Big(-c'xy\Big)\\
			&\lesi \f{1}{e^t}\exp\Big(-\f{|x-y|^2}{ct}\Big) \Big(\f{1}{xy}\Big)^{\gamma_\nu},
		\end{aligned}
	\end{equation}
	which implies \eqref{eq-ptxy nu < -1/2}.
	
	This completes our proof.
\end{proof}

The following gives another upper bound for the heat kernel $p_t^\nu(x,y)$ which plays an essential role in the proof of the boundedness of the maximal operator.
\begin{lem}\label{lem 1}
	Let $\nu \in (-1,-1/2)$. Then there exists $c>0$ such that 
	\[
	p_t^\nu(x,y) \lesi \f{1}{\sqrt t}\exp\Big(-\f{|x-y|^2}{ct}\Big)+ \f{1}{x}\Big(\f{x}{y}\Big)^{\gamma_\nu}\chi_{\{y< x/2\}} +\f{1}{y}\Big(\f{y}{x}\Big)^{\gamma_\nu}\chi_{\{x/2\le y\}}
	\]
	for  $t>0$ and $x,y>0$.
\end{lem}
\begin{proof}
	\textbf{Case 1: $t\in (0,1)$.} In this situation,
	\[
	p_t^\nu(x,y)\le p_t^\nu(x,y)\chi_{\{xy\ge t\}}+  p_t^\nu(x,y)\chi_{\{xy< t\}}
	\]
	It is obvious that due to \eqref{ptxy subcase 1.2},
	\[
	p_t^\nu(x,y)\chi_{\{xy\ge t\}} \lesi \f{1}{\sqrt t}\exp\Big(-\f{|x-y|^2}{ct}\Big).
	\]
	For the second term, from \eqref{eq-est1-proof Prop1}, we have
	\[
	\begin{aligned}
		p_t^\nu(x,y)\chi_{\{xy< t\}}&\lesi \f{1}{\sqrt t}\exp\Big(-\f{|x-y|^2}{ct}\Big)\Big(\f{\sqrt t}{x}\Big)^{\gamma_\nu}\Big(\f{\sqrt t}{y}\Big)^{\gamma_\nu}\chi_{\{xy< t\}}\\
		&=\f{1}{\sqrt t}\exp\Big(-\f{|x-y|^2}{ct}\Big)\Big(\f{\sqrt t}{x}\Big)^{\gamma_\nu}\Big(\f{\sqrt t}{y}\Big)^{\gamma_\nu}\chi_{\{xy< t\}\cap\{y< x/2\}}\\
		& \ \ + \f{1}{\sqrt t}\exp\Big(-\f{|x-y|^2}{ct}\Big)\Big(\f{\sqrt t}{x}\Big)^{\gamma_\nu}\Big(\f{\sqrt t}{y}\Big)^{\gamma_\nu}\chi_{\{xy< t\}\cap\{y>2x\}}\\
		& \ \ + \f{1}{\sqrt t}\exp\Big(-\f{|x-y|^2}{ct}\Big)\Big(\f{\sqrt t}{x}\Big)^{\gamma_\nu}\Big(\f{\sqrt t}{y}\Big)^{\gamma_\nu}\chi_{\{xy< t\}\cap\{x/2\le y\le 2x\}}\\
		&=E_1 + E_2 + E_3.
	\end{aligned}
	\]
	For $E_1$, we have $|x-y|\sim x$ whenever $y\le x/2$. Hence,
	\[
	\begin{aligned}
		E_1&\lesi \f{1}{\sqrt t} \exp\Big(-\f{x^2}{ct}\Big) \Big(\f{\sqrt t}{x}\Big)^{\gamma_\nu}\Big(\f{\sqrt t}{y}\Big)^{\gamma_\nu}\chi_{\{xy< t\}\cap\{y< x/2\}}\\
		&\sim \f{1}{t^{(1-2\gamma_\nu)/2}} \exp\Big(-\f{x^2}{ct}\Big) \Big(\f{1}{x}\Big)^{\gamma_\nu}\Big(\f{1}{y}\Big)^{\gamma_\nu}\chi_{\{xy< t\}\cap\{y< x/2\}}\\
		&\lesi \f{1}{x^{ 1-2\gamma_\nu }}  \Big(\f{1}{x}\Big)^{\gamma_\nu}\Big(\f{1}{y}\Big)^{\gamma_\nu}\chi_{\{xy< t\}\cap\{y< x/2\}}\\
		&\sim \f{1}{x}\Big(\f{x}{y}\Big)^{\gamma_\nu}\chi_{\{y< x/2\}}.
	\end{aligned}
	\] 
	Similarly,
	\[
	E_2\lesi  \f{1}{y}\Big(\f{y}{x}\Big)^{\gamma_\nu}\chi_{\{y>2x\}}.
	\]
	
	For the last term $E_3$, if $x/2<y\le 2x$, then $x\sim y$. In this case,
	\[
	\begin{aligned}
		E_3&\lesi \f{1}{\sqrt t}  \Big(\f{\sqrt t}{x}\Big)^{\gamma_\nu}\Big(\f{\sqrt t}{y}\Big)^{\gamma_\nu}\chi_{\{xy< t\}\cap\{x/2\le y\le 2x\}}\sim \f{1}{t^{(1-2\gamma_\nu)/2}}   \Big(\f{1}{x}\Big)^{\gamma_\nu}\Big(\f{1}{y}\Big)^{2\gamma_\nu}\chi_{\{xy< t\}\cap\{x/2\le y\le 2x\}}\\
		&\lesi \f{1}{(xy)^{ (1-2\gamma_\nu)/2 }}  \Big(\f{1}{x}\Big)^{\gamma_\nu}\Big(\f{1}{y}\Big)^{\gamma_\nu}\chi_{\{xy< t\}\cap\{x/2\le y\le 2x\}}\\
		&\sim \f{1}{y^{  1-2\gamma_\nu  }}  \Big(\f{1}{x}\Big)^{\gamma_\nu}\Big(\f{1}{y}\Big)^{\gamma_\nu}\chi_{\{xy< t\}\cap\{x/2\le y\le 2x\}}\\
		&\sim \f{1}{y}\Big(\f{y}{x}\Big)^{\gamma_\nu}\chi_{\{x/2\le y\le 2x\}},
	\end{aligned}
	\]
	where in the second inequality we used $(xy)^{(1-2\gamma_\nu)/2}\le  t^{(1-2\gamma_\nu)/2}$ since $1-2\gamma_\nu>0$ and $t>xy$.
	
	\bigskip
	
	\textbf{Case 2: $t\ge 1$.} In this situation,
	\[
	p_t^\nu(x,y)\le p_t^\nu(x,y)\chi_{\{xy\ge e^{2t}\}}+  p_t^\nu(x,y)\chi_{\{xy< e^{2t}\}}
	\]
	From \eqref{ptxy subcase 2.2},
	\[
	p_t^\nu(x,y)\chi_{\{xy\ge e^{2t}\}} \lesi \f{1}{\sqrt t}\exp\Big(-\f{|x-y|^2}{ct}\Big).
	\]
	For the second term, by using \eqref{eq- heat kernel xy< e2t}, we can write
	\[
	\begin{aligned}
		p_t^\nu(x,y)\chi_{\{xy< e^{2t}\}}&\lesi \f{1}{e^t}\exp\Big(-\f{|x-y|^2}{ce^{2t}}\Big)\Big(\f{e^t}{x}\Big)^{\gamma_\nu}\Big(\f{e^t}{y}\Big)^{\gamma_\nu}\chi_{\{xy< e^{2t}\}}\\
		&=\f{1}{e^t}\exp\Big(-\f{|x-y|^2}{ce^{2t}}\Big)\Big(\f{e^t}{x}\Big)^{\gamma_\nu}\Big(\f{e^t}{y}\Big)^{\gamma_\nu}\chi_{\{xy< e^{2t}\}\cap\{y< x/2\}}\\
		& \ \ + \f{1}{e^t}\exp\Big(-\f{|x-y|^2}{ce^{2t}}\Big)\Big(\f{e^t}{x}\Big)^{\gamma_\nu}\Big(\f{e^t}{y}\Big)^{\gamma_\nu}\chi_{\{xy< e^{2t}\}\cap\{y>2x\}}\\
		& \ \ + \f{1}{e^t}\exp\Big(-\f{|x-y|^2}{ce^{2t}}\Big)\Big(\f{e^t}{x}\Big)^{\gamma_\nu}\Big(\f{e^t}{y}\Big)^{\gamma_\nu}\chi_{\{xy< e^{2t}\}\cap\{x/2\le y\le 2x\}}\\
		&=F_1 + F_2 + F_3.
	\end{aligned}
	\]
	For $E_1$, we have $|x-y|\sim x$ whenever $y\le x/2$. Hence,
	\[
	\begin{aligned}
		F_1&\lesi \f{1}{e^t} \exp\Big(-\f{x^2}{ce^{2t}}\Big) \Big(\f{e^t}{x}\Big)^{\gamma_\nu}\Big(\f{e^t}{y}\Big)^{\gamma_\nu}\chi_{\{xy< e^{2t}\}\cap\{y< x/2\}}\\
		&\sim \f{1}{e^{t(1-2\gamma_\nu)}} \exp\Big(-\f{x^2}{ce^{2t}}\Big) \Big(\f{1}{x}\Big)^{\gamma_\nu}\Big(\f{1}{y}\Big)^{\gamma_\nu}\chi_{\{xy< e^{2t}\}\cap\{y< x/2\}}\\
		&\lesi \f{1}{x^{ 1-2\gamma_\nu }}  \Big(\f{1}{x}\Big)^{\gamma_\nu}\Big(\f{1}{y}\Big)^{\gamma_\nu}\chi_{\{xy< e^{2t}\}\cap\{y< x/2\}}\\
		&\sim \f{1}{x}\Big(\f{x}{y}\Big)^{\gamma_\nu}\chi_{\{y< x/2\}}.
	\end{aligned}
	\] 
	Similarly,
	\[
	E_2\lesi  \f{1}{y}\Big(\f{y}{x}\Big)^{\gamma_\nu}\chi_{\{y>2x\}}.
	\]
	
	For the last term $E_3$, if $x/2<y\le 2x$, then we have
	\[
	\begin{aligned}
		E_3&\lesi \f{1}{e^t}  \Big(\f{e^t}{x}\Big)^{\gamma_\nu}\Big(\f{e^t}{y}\Big)^{\gamma_\nu}\chi_{\{xy< e^{2t}\}\cap\{x/2\le y\le 2x\}}\sim \f{1}{e^{t(1-2\gamma_\nu)}}   \Big(\f{1}{x}\Big)^{\gamma_\nu}\Big(\f{1}{y}\Big)^{2\gamma_\nu}\chi_{\{xy< e^{2t}\}\cap\{x/2\le y\le 2x\}}\\
		&\lesi \f{1}{(xy)^{ (1-2\gamma_\nu)/2 }}  \Big(\f{1}{x}\Big)^{\gamma_\nu}\Big(\f{1}{y}\Big)^{\gamma_\nu}\chi_{\{xy< e^{2t}\}\cap\{x/2\le y\le 2x\}}\\
		&\sim \f{1}{y^{  1-2\gamma_\nu  }}  \Big(\f{1}{x}\Big)^{\gamma_\nu}\Big(\f{1}{y}\Big)^{\gamma_\nu}\chi_{\{xy< e^{2t}\}\cap\{x/2\le y\le 2x\}}\\
		&\sim \f{1}{y}\Big(\f{y}{x}\Big)^{\gamma_\nu}\chi_{\{x/2\le y\le 2x\}},
	\end{aligned}
	\]
	where in the second inequality we used $(xy)^{(1-2\gamma_\nu)/2}\le e^{t(1-2\gamma_\nu)}$.
	
	This completes our proof.
\end{proof}

\bigskip

We now take care of the estimate for $\delta p_t^\nu(x,y)$. To do this, from \eqref{eq1-ptxy} we can rewrite $p_t^\nu(x,y)$ as
\begin{equation}\label{eq- form ptxy for derivative}
p_t^\nu(x,y)=\Big(\f{2\sqrt r}{1-r}\Big)^{1/2}\Big(\f{2r^{1/2}}{1-r}xy\Big)^{\nu+1/2}\exp\Big(-\f{1}{2}\f{1+r}{1-r}(x^2+y^2)\Big)\Big(\f{2r^{1/2}}{1-r}xy\Big)^{-\nu}I_\nu\Big(\f{2r^{1/2}}{1-r}xy\Big),
\end{equation}
where $r=e^{-4t}$.

Setting 
$$
H_\nu(r;x,y)=\Big(\f{2r^{1/2}}{1-r}xy\Big)^{-\nu}I_\nu\Big(\f{2r^{1/2}}{1-r}xy\Big),
$$
then
\begin{equation}\label{eq-new formula of heat kernel}
	p_t^\nu(x,y)=\Big(\f{2r^{1/2}}{1-r}\Big)^{1/2}\Big(\f{2r^{1/2}}{1-r}xy\Big)^{\nu+1/2}\exp\Big(-\f{1}{2}\f{1+r}{1-r}(x^2+y^2)\Big)H_\nu(r;x,y). 
\end{equation}
From \eqref{eq-new formula of heat kernel}, applying the chain rule,
\begin{equation}\label{eq- chain rule}
	\begin{aligned}
		\partial_x p_t^\nu(x,y)&=  \f{\nu+1/2}{x}p_t^\nu(x,y)- \f{1+r}{1-r}xp_t^\nu(x,y) + \f{2r^{1/2}}{1-r}y p_t^{\nu+1}(x,y),
	\end{aligned}
\end{equation}
which implies
\begin{equation}\label{eq- chain rule 2}
	\begin{aligned}
		\delta p_t^\nu(x,y)&=  xp_t^\nu(x,y)- \f{1+r}{1-r}xp_t^\nu(x,y) + \f{2r^{1/2}}{1-r}y p_t^{\nu+1}(x,y).
	\end{aligned}
\end{equation}
On the other hand, since $\partial (xf) = f + \partial f$, we have
\begin{equation}
	\label{eq-formula for delta m}
	\delta(xf) = f + \delta f.
\end{equation}

\begin{prop}
	\label{prop- delta k pt} Let $\nu>-1$. Then we have
	\begin{equation}
		\label{eq- delta k pt}
		| \delta   p_t^\nu(x,y)|\lesi   \f{1}{t}\exp\Big(-\f{|x-y|^2}{ct}\Big) \Big(1+ \f{\sqrt t}{y}\Big)^{\gamma_\nu}
	\end{equation}
	for $t>0$ and $x,y\in (0,\vc)$.
\end{prop}
\begin{proof}
	For $\nu\ge -1/2$, the estimate \eqref{eq- delta k pt} is a direct consequence of \cite[Propositions 3.2 \& 3.6]{B}. We need only to take case of the case $-1<\nu<-1/2$. To do this, we consider two cases.

		\noindent\textbf{Case 1: $t \in (0,1)$.}  
		We also have two subcases.
	
	\textbf{Subcase 1.1: $xy<t$.} From \eqref{eq- chain rule 2} and the facts that $r\sim 1 + r \sim 1 $ and $1-r\sim t$, we have
	\[
	\begin{aligned}
		\delta p_t^\nu(x,y)&\lesi   xp_t^\nu(x,y)+ \f{x}{t}p_t^\nu(x,y) + \f{y}{t} p_t^{\nu+1}(x,y)\\
		&\lesi   \f{x}{t}p_t^\nu(x,y) + \f{y}{t} p_t^{\nu+1}(x,y).
	\end{aligned}
	\]
	Using \eqref{eq-est1-proof Prop1},
	\begin{equation*}
		\begin{aligned}
			\f{x}{t}p_t^\nu(x,y)&\lesi \f{x}{t}\f{1}{\sqrt t}\exp\Big(-\f{x^2+y^2}{ct}\Big)\Big(\f{t}{xy}\Big)^{\gamma_\nu}\\
			&\lesi \Big(\f{x}{\sqrt t}\Big)^{1-\gamma_\nu}\f{1}{t}\exp\Big(-\f{x^2+y^2}{ct}\Big)\Big(\f{\sqrt t}{y}\Big)^{\gamma_\nu}\\
			&\lesi  \f{1}{ t}\exp\Big(-\f{x^2+y^2}{2ct}\Big)\Big(\f{\sqrt t}{y}\Big)^{\gamma_\nu}.
		\end{aligned}
	\end{equation*}
	In addition, By \eqref{eq1-ptxy} and \eqref{eq1-Inu},
	\[
		\begin{aligned}
			\f{y}{t}p_t^{\nu+1}(x,y)&\lesi \f{y}{\sqrt t}\f{1}{ t}\exp\Big(-\f{x^2+y^2}{ct}\Big)\Big(\f{xy}{t}\Big)^{\nu+3/2}\\
			&\lesi  \f{1}{ t}\exp\Big(-\f{x^2+y^2}{2ct}\Big)\\
			&\lesi  \f{1}{ t}\exp\Big(-\f{x^2+y^2}{2ct}\Big)\Big(\f{\sqrt t}{y}\Big)^{\gamma_\nu}.
		\end{aligned}		
	\]
	These two estimates proves \eqref{eq- delta k pt}.
	
	\textbf{Subcase 1.2: $xy\ge t$.} From \eqref{eq- chain rule 2}, we rewrite as follows
	\[
	\begin{aligned}
		|\delta p_t^\nu(x,y)|&=  \Big|xp_t^\nu(x,y)- \f{1+r}{1-r}xp_t^\nu(x,y) +\f{2r^{1/2}}{1-r}x p_t^{\nu+1}(x,y)+ \f{2r^{1/2}}{1-r}(y-x) p_t^{\nu+1}(x,y)\Big|\\
		&\lesi xp_t^\nu(x,y) + \Big|\f{1+r}{1-r}xp_t^\nu(x,y) -\f{2r^{1/2}}{1-r}x p_t^{\nu+1}(x,y)\Big| + \f{|x-y|}{t}p_t^{\nu+1}(x,y)\\
		&=: E_1 + E_2 + E_3.
	\end{aligned}
	\]
	The term $E_3$ is straightforward. Indeed, from Proposition \ref{prop-heat kernel},
	\[
	\begin{aligned}
		E_3&\lesi \f{|x-y|}{t}p_t^{\nu+1}(x,y)\\
		&\lesi \f{|x-y|}{t}\f{1}{ \sqrt t}\exp\Big(-\f{|x-y|^2}{ct}\Big)\\
		&\lesi \f{|x-y|}{t}\f{1}{ \sqrt t}\f{\sqrt t}{|x-y|}\exp\Big(-\f{|x-y|^2}{2ct}\Big)\tyg\\
		&\lesi  \f{1}{   t}\exp\Big(-\f{|x-y|^2}{2ct}\Big)\tyg.
	\end{aligned}
	\]
	
	We now take care of $E_1$.  If $y<x/2$, then $x\gtrsim \sqrt t$ and $|x-y|>x/2$. Applying Proposition \ref{prop-heat kernel},
	\[
	\begin{aligned}
		xp_t^\nu(x,y)&\lesi  \f{x}{\sqrt t}\exp\Big(-\f{|x-y|^2}{ct}\Big) \tyg\\
		&\lesi    \f{x}{\sqrt t}\f{\sqrt t}{|x-y|}\exp\Big(-\f{|x-y|^2}{2ct}\Big) \tyg\\
		&\lesi \f{1}{t}\exp\Big(-\f{|x-y|^2}{2ct}\Big)\tyg.
	\end{aligned}
	\]
	If $y> 2x$ and $x\ge \sqrt t$, similarly to above, since we also have $|x-y|>x/2$, we have
	\[
	\begin{aligned}
		xp_t^\nu(x,y)
		&\lesi \f{1}{t}\exp\Big(-\f{|x-y|^2}{2ct}\Big)\tyg.
	\end{aligned}
	\]
	If $y> 2x$ and $x< \sqrt t$, then from Proposition \ref{prop-heat kernel},
	\[
	\begin{aligned}
		xp_t^\nu(x,y)&\lesi  \f{x}{\sqrt t}\exp\Big(-\f{|x-y|^2}{ct}\Big)\Big(\f{\sqrt t}{x}\Big)^{\gamma_\nu}\tyg\\
		&\lesi    \Big(\f{x}{\sqrt t}\Big)^{1-\gamma_\nu}\exp\Big(-\f{|x-y|^2}{2ct}\Big)\tyg\\
		&\lesi \exp\Big(-\f{|x-y|^2}{2ct}\Big)\tyg\\
		&\lesi \f{1}{t}\exp\Big(-\f{|x-y|^2}{2ct}\Big)\tyg
	\end{aligned}
	\]
	Otherwise, if $x/2\le y\le 2x$,  then $x\sim y\gtrsim \sqrt t$ since $xy\ge t$. In this case, by the first inequality in 	 \eqref{ptxy subcase 1.2}, we have\[
	p_t^\nu(x,y)\le \f{C}{\sqrt{t}}\exp\Big(-\f{|x-y|^2}{ct}\Big)\exp\Big(-c'txy\Big),
	\]
	which implies
	\[
	\begin{aligned}
		xp_t^\nu(x,y)&\lesi \f{x}{\sqrt{t}}\exp\Big(-\f{|x-y|^2}{ct}\Big) \Big(\f{1}{\sqrt {txy}}\Big)\\
		&\sim  \f{x}{\sqrt{t}}\exp\Big(-\f{|x-y|^2}{ct}\Big) \Big(\f{1}{\sqrt {t}x}\Big)\\
		&\sim  \f{1}{t}\exp\Big(-\f{|x-y|^2}{ct}\Big) \\
		&\sim  \f{1}{t}\exp\Big(-\f{|x-y|^2}{ct}\Big)  \Big(1+ \f{\sqrt t}{y}\Big)^{\gamma_\nu}.
	\end{aligned}
	\]
	
	It remains to estimate $E_2$. To do this we write
	\[
	\begin{aligned}
		E_2 &= \Big|\f{1+r}{1-r}x[p_t^\nu(x,y)-p_t^{\nu+1}(x,y)] +\f{1-2r^{1/2}+r}{1-r}x p_t^{\nu+1}(x,y)\Big|\\
		&\le \f{1+r}{1-r}x\big|p_t^\nu(x,y)-p_t^{\nu+1}(x,y)\big| + \f{(1-\sqrt r)^{1/2}}{1-r}x p_t^{\nu+1}(x,y)\\
		&\sim \f{1+r}{1-r}x\big|p_t^\nu(x,y)-p_t^{\nu+1}(x,y)\big| + \f{1-\sqrt r}{1+\sqrt r}x p_t^{\nu+1}(x,y)\\
		&\lesi \f{x}{t}\big|p_t^\nu(x,y)-p_t^{\nu+1}(x,y)\big| + txp_t^{\nu+1}(x,y)=: E_{21} + E_{22}.	
	\end{aligned}
	\]
	where in the last inequality we used $1-\sqrt r \sim 
	t$ and $1+r\sim 1$.
	
	Similarly to $E_1$, we have
	\[
	E_{22}\lesi \f{1}{t}\exp\Big(-\f{|x-y|^2}{ct}\Big)\Big(1+ \f{\sqrt t}{y}\Big)^{\gamma_\nu}.
	\]
	For the term $E_{21}$, using \eqref{eq5-Inu} and \eqref{eq1-ptxy}, we have
	\[
	\begin{aligned}
		\big|p_t^\nu(x,y)-p_t^{\nu+1}(x,y)\big|&\lesi \f{1-r}{2\sqrt r xy}p_t^{\nu+1}(x,y)\\
		&\sim \f{t}{xy}p_t^{\nu+1}(x,y).
	\end{aligned}
	\]
	which implies 
	\[
	E_{21}\lesi \f{t}{xy}\frac{x}{t}p_t^{\nu+1}(x,y).
	\]
	If $y\le x/2$ or $y\ge 2x$, then $|x-y|\gtrsim x$. Hence,  by Proposition \ref{prop-heat kernel} and the fact $xy\ge t$,
	\[
	\begin{aligned}
		E_{21}&\lesi  \f{x}{t}\f{1}{\sqrt t}\exp\Big(-\f{|x-y|^2}{ct}\Big)\\
		&\lesi  \f{x}{t}\f{\sqrt t}{|x-y|}\f{1}{\sqrt t}\exp\Big(-\f{|x-y|^2}{2ct}\Big)\\
		&\lesi  \f{1}{\sqrt t}\exp\Big(-\f{|x-y|^2}{2ct}\Big)\\
		&\lesi \f{1}{t}\exp\Big(-\f{|x-y|^2}{2ct}\Big)\tyg.
	\end{aligned} 
	\]
	
	If $x/2<y<2x$, then $x\sim y\gtrsim \sqrt t$ since $xy\ge t$. Hence, by Proposition \ref{prop-heat kernel},
	\[
	\begin{aligned}
		E_{21}&\lesi \f{t}{xy}\f{x}{t}\f{1}{\sqrt t}\exp\Big(-\f{|x-y|^2}{ct}\Big) \\
		&\lesi \f{1}{y}\f{1}{\sqrt t}\exp\Big(-\f{|x-y|^2}{ct}\Big)\\
		&\lesi  \f{1}{t}\exp\Big(-\f{|x-y|^2}{ct}\Big)\\
		&\lesi \f{1}{t}\exp\Big(-\f{|x-y|^2}{ct}\Big)\Big(1+ \f{\sqrt t}{y}\Big)^{\gamma_\nu}.
	\end{aligned} 
	\]
	
\bigskip	
	
	\noindent\textbf{Case 2: $t\ge 1$.}	In this case, $1+r\sim 1-r\sim 1$ and $r\le 1$. Hence,
	\[
	\begin{aligned}
		|\delta p_t^\nu(x,y)|\lesi xp_t^\nu(x,y) +  y p_t^{\nu+1}(x,y).
	\end{aligned}
	\]
	\textbf{Subcase 2.1: $xy<e^{2t}$.} By \eqref{eq1-ptxy}, \eqref{eq1-Inu}  and $1-2\gamma_\nu>0$ we have
	\begin{equation*}
		\begin{aligned}
			xp_t^\nu(x,y)&\lesi x\f{1}{e^{t}}\exp\Big(-c(x^2+y^2)\Big)\Big(\f{1}{e^{-2t}xy}\Big)^{\gamma_\nu}\\
			&\lesi \f{x^{1-\gamma_\nu}}{e^{t(1-2\gamma_\nu)}}\exp\Big(-c(x^2+y^2)\Big)\Big(\f{1}{y}\Big)^{\gamma_\nu}\\
			&\lesi \f{1}{t}\exp\Big(-c(x^2+y^2)/2\Big)\Big(\f{\sqrt t}{y}\Big)^{\gamma_\nu}\\
		\end{aligned}
	\end{equation*}
	
	which implies
	\begin{equation*}
	\begin{aligned}
		xp_t^\nu(x,y)&\lesi \f{1}{t}\exp\Big(-\f{|x-y|^2}{ct}\Big) \Big(1+ \f{\sqrt t}{y}\Big)^{\gamma_\nu},
	\end{aligned}
	\end{equation*}
	as desired.
	
	For the second term, by \eqref{eq1-ptxy}, \eqref{eq1-Inu}  we have
	\begin{equation*}
		\begin{aligned}
			yp_t^{\nu+1}(x,y)&\lesi y\f{1}{e^{t}}\exp\Big(-c(x^2+y^2)\Big)[e^{-2t}xy]^{\nu+3/2}\\
			&\lesi \f{1}{\sqrt te^{t}}\exp\Big(-c(x^2+y^2)/2\Big) \\
			&\lesi \f{1}{t}\exp\Big(-c(x^2+y^2)/2\Big)\Big(\f{\sqrt t}{y}\Big)^{\gamma_\nu}.
		\end{aligned}
	\end{equation*}
	
	\medskip
	
	\textbf{Subcase 2.2: $xy\ge e^{2t}$.} From the first inequality of  \eqref{ptxy subcase 2.2},
	\begin{equation}
		\label{ptxy subcase 2.2 s}
		\begin{aligned}
		p_t^\nu(x,y)&\lesi \f{1}{e^t}\exp\Big(-c|x-y|^2\Big)\exp\Big(-c'xy\Big)\\
		&\lesi \f{1}{e^t}\exp\Big(-c|x-y|^2\Big)\Big(\f{1}{xy}\Big)^{\gamma_\nu}\\
		&\lesi \f{1}{e^t}\exp\Big(-c|x-y|^2\Big)\Big(\f{t}{xy}\Big)^{\gamma_\nu}.
	\end{aligned}
\end{equation}

	Since $|x-y|+y\ge x$, either $|x-y|>x/2$ or $y>x/2$. If $|x-y|>x/2$,  then  we have
	\[
	\begin{aligned}
		xp_t^\nu(x,y)&\lesi \f{x}{e^{t}}\exp\Big(-c|x-y|^2\Big)\Big(\f{1}{|x-y|}\Big)^{1-\gamma_\nu}\Big(\f{t}{xy}\Big)^{\gamma_\nu} \\
		&\sim \f{t^{\gamma_\nu/2}x^{1-\gamma_\nu}}{e^{t}}\exp\Big(-c|x-y|^2\Big)^{1-\gamma_\nu}\Big(\f{1}{|x-y|}\Big)\Big(\f{\sqrt t}{y}\Big)^{\gamma_\nu} \\
		&\lesi \f{t^{\gamma_\nu/2}}{e^{t}}\exp\Big(-\f{|x-y|^2}{ct}\Big)\Big(\f{\sqrt t}{y}\Big)^{\gamma_\nu}\\
		&\lesi \f{1}{t}\exp\Big(-\f{|x-y|^2}{ct}\Big)\tyg.
	\end{aligned}
	\]
	
	If $y>x/2$, then from the first inequality of \eqref{ptxy subcase 2.2} we have
	\[
	\begin{aligned}
		x p_t^\nu(x,y)&\lesi \f{x }{e^t}\exp\Big(-\f{|x-y|^2}{ct}\Big)\Big(\f{1}{xy}\Big)^{\gamma_\nu/2+1/2}\\
		&\lesi \f{1}{e^{t}}\exp\Big(-\f{|x-y|^2}{ct}\Big)\Big(\f{1}{y}\Big)^{\gamma_\nu}\\
		&\lesi \f{1}{t}\exp\Big(-\f{|x-y|^2}{ct}\Big)\tyg.
	\end{aligned}
	\]
	For the reaming term $yp^{\nu+1}_t(x,y)$, we write
	\[
	yp^{\nu+1}_t(x,y) = (y-x)p^{\nu+1}_t(x,y)+xp^{\nu+1}_t(x,y).
	\]
	The term $xp^{\nu+1}_t(x,y)$ follows by using the above argument while the term $(y-x)p^{\nu+1}_t(x,y)$ follows from Proposition \ref{prop-heat kernel}.

	\bigskip

	This completes our proof.
\end{proof}

\begin{prop}
	\label{prop-delta dual heat kernel}
	Let $\nu>-1$. Then 
		\[
		| \delta^*p^{\nu+1}_t(x,y)|\lesi \f{1}{t}\exp\Big(-\f{|x-y|^2}{ct}\Big)\Big(1+\f{\sqrt t}{x} \Big)^{\gamma_\nu}
		\]
		for all $t>0$ and $x,y\in (0,\vc)$.
	
	\end{prop}
\begin{proof}
	The estimate corresponding to $\nu\ge -1/2$ is a direct consequence of \cite[Propositions 3.2 \& 3.8 ]{B}. It remains to prove the estimate when $-1<\nu<-1/2$. To do this, denote 
	\[
	\widetilde{ \delta} =  \partial_x  + x-\frac{1}{x}\Big(\nu+1 + \f{1}{2}\Big).
	\]
	Then we have
	\[
	\begin{aligned}
		\delta^* &=: -\widetilde{ \delta} +2x - \f{1}{x}\Big(2\nu+1 + 1\Big).
	\end{aligned}
	\]
	Hence,
	\[
	\begin{aligned}
		|\delta^*p^{\nu+1}_t(x,y)|&\lesi |\widetilde{\delta} p^{\nu+1}_t(x,y)| + x p^{\nu+1}_t(x,y) + \f{1}{x}p^{\nu+1}_t(x,y).
	\end{aligned}
	\]

	By Proposition \ref{prop- delta k pt} in \cite{B},
	\[
	\begin{aligned}
		|\widetilde{\delta} p^{\nu+1}_t(x,y)|&\lesi \f{1}{t}\exp\Big(-\f{|x-y|^2}{ct}\Big)\\
		&\lesi \f{1}{t}\exp\Big(-\f{|x-y|^2}{ct}\Big)\txg.
	\end{aligned}
	\]
	We now show that 
	\[
	xp_t^{\nu+1}(x,y)\lesi \f{1}{t}\exp\Big(-\f{|x-y|^2}{ct}\Big)\txg.
	\]
	If $x>2y$ or $x< y/2$, then $|x-y|\gtrsim x$. This, together with Proposition \ref{prop-heat kernel} (i), implies
	\[
	\begin{aligned}
		xp_t^{\nu+1}(x,y)&\lesi \f{x}{\sqrt t}\exp\Big(-\f{|x-y|^2}{ct}\Big)\\
		&\lesi \f{x}{\sqrt t}\f{\sqrt t}{|x-y|}\exp\Big(-\f{|x-y|^2}{2ct}\Big)\\
		&\lesi \f{1}{t} \exp\Big(-\f{|x-y|^2}{2ct}\Big)\\
		&\lesi \f{1}{t} \exp\Big(-\f{|x-y|^2}{2ct}\Big)\txg.
	\end{aligned}
	\]
	If $y/2\le x\le 2y$, by a careful examination the proof of Proposition \ref{prop- delta k pt}  we have
	\[
	\begin{aligned}
		xp_t^{\nu+1}(x,y) &\lesi\f{1}{t} \exp\Big(-\f{|x-y|^2}{2ct}\Big)\tyg\\
		&\sim \f{1}{t} \exp\Big(-\f{|x-y|^2}{2ct}\Big)\txg.
	\end{aligned}
	\]
We have proved that 
\[
		xp_t^{\nu+1}(x,y)\lesi  \f{1}{t} \exp\Big(-\f{|x-y|^2}{2ct}\Big)\txg.
\]
	It remains to show that 
	\[
	\f{1}{x}p_t^{\nu+1}(x,y)\lesi \f{1}{t}\exp\Big(-\f{|x-y|^2}{ct}\Big)\tyg.
	\]
	To do this, we consider two cases.
	
			\noindent\textbf{Case 1: $t \in (0,1)$.}  If $x\ge \sqrt t$, then by Proposition \ref{prop-heat kernel},
	\[
		\begin{aligned}
		\f{1}{x}p_t^{\nu+1}(x,y)&\lesi \f{1}{\sqrt t }p_t^{\nu+1}(x,y) \\
		&\lesi \f{1}{t}\exp\Big(-\f{|x-y|^2}{ct}\Big) \\
		&\lesi \f{1}{t}\exp\Big(-\f{|x-y|^2}{ct}\Big)\txg.
	\end{aligned}	
	\]
	If $x<\sqrt t \le 1$, then apply \eqref{eq-ptxy nu > -1/2} with the consideration  that $\rho(x) \sim x$,
	\begin{equation*}
		\begin{aligned}
			\f{1}{x}p_t^{\nu+1}(x,y)&\lesi \f{1}{x\sqrt t}\exp\Big(-\f{|x-y|^2}{ct}\Big)\Big(\f{x}{\sqrt t}\Big)^{\nu+1+1/2}\\
			&\lesi \f{1}{t}\exp\Big(-\f{|x-y|^2}{ct}\Big)\Big(\f{x}{\sqrt t}\Big)^{\nu+1/2}\\
			&\lesi \f{1}{t}\exp\Big(-\f{|x-y|^2}{ct}\Big)\txg.
		\end{aligned}
	\end{equation*}

\bigskip

\noindent\textbf{Case 2: $t> 1$.} If $x\ge 1$, then by \eqref{eq-ptxy nu > -1/2},
\[
\begin{aligned}
	\f{1}{x}p_t^{\nu+1}(x,y)&\lesi p_t^{\nu+1}(x,y)\\ 
	&\lesi \f{e^{-t/2}}{\sqrt t}\exp\Big(-\f{|x-y|^2}{ct}\Big)\\
	&\lesi \f{1}{t}\exp\Big(-\f{|x-y|^2}{ct}\Big)\txg.	
\end{aligned}
\]
If $x< 1$, then apply \eqref{eq-ptxy nu > -1/2} with the consideration  that $\rho(x) \sim x$,
\begin{equation*}
	\begin{aligned}
		\f{1}{x}p_t^{\nu+1}(x,y)&\lesi \f{1}{x\sqrt t}\exp\Big(-\f{|x-y|^2}{ct}\Big)\Big(\f{x}{\sqrt t}\Big)^{\nu+1+1/2}\\
		&\lesi \f{1}{t}\exp\Big(-\f{|x-y|^2}{ct}\Big)\Big(\f{x}{\sqrt t}\Big)^{\nu+1/2}\\
		&\sim \f{1}{t}\exp\Big(-\f{|x-y|^2}{ct}\Big)\txg.
	\end{aligned}
\end{equation*}
	This completes our proof.
\end{proof}

The following result concerns estimates of the time derivative of the heat kernel $p_t^\nu(x,y)$.
\begin{prop}\label{prop- time derivative of heat kernel}
		Let $ \nu>-1$. Then we have
	\[
	| \partial_tp^{\nu}_t(x,y)|\lesi \f{1}{t^{3/2}}\exp\Big(-\f{|x-y|^2}{ct}\Big)\Big(1+\f{\sqrt t}{x} \Big)^{\gamma_\nu}\Big(1+\f{\sqrt t}{y} \Big)^{\gamma_\nu}
	\]
	for all $t>0$ and $x,y\in (0,\vc)$.
\end{prop}
\begin{proof}
	If $\nu\ge -1/2$, then Proposition \ref{prop-heat kernel} yields
	\[
	| p^{\nu}_t(x,y)|\lesi \f{1}{\sqrt t}\exp\Big(-\f{|x-y|^2}{ct}\Big).
	\]
	Applying Lemma 2.5 in \cite{CD}, we have
	\[
	\begin{aligned}
		| \partial_tp^{\nu}_t(x,y)|&\lesi \f{1}{t^{3/2}}\exp\Big(-\f{|x-y|^2}{ct}\Big),
	\end{aligned}
	\]
	as desired since $\gamma_\nu = 0$.
	
	We now take care of the case $-1<\nu<-1/2$. From \eqref{eq- form ptxy for derivative}, by the product rule and the chain rule, we have
	\[
	\begin{aligned}
		\partial_t p^\nu_t(x,y)&= (\nu+1)\partial_t\Big[\f{r^{1/2}}{1-r}\Big]\times \Big[\f{r^{1/2}}{1-r}\Big]^{-1}p^\nu_t(x,y) -\f{1}{2}\partial_t\Big[\f{1+r}{1-r}\Big](x^2+y^2) p^\nu_t(x,y)\\
		&+2\partial_t\Big[\f{r^{1/2}}{1-r}\Big]xyp_t^{\nu+1}(x,y),
	\end{aligned}
	\]
	where $r=e^{-4t}$.
	
	By a simple calculation, we come up with
	\begin{equation}\label{eq-time derivative of ptnu}
	\begin{aligned}
		\partial_t p^\nu_t(x,y)&= -\f{2(\nu+1)(1+r)}{1-r}p^\nu_t(x,y) +\f{4r}{(1-r)^2}(x^2+y^2) p^\nu_t(x,y)-\f{4r^{1/2}(1+r)}{(1-r)^2}xyp_t^{\nu+1}(x,y).
	\end{aligned}
	\end{equation}
	We now consider two cases.
	
	\bigskip
	
	\textbf{Case 1: $t\le 1$.} We have two subcases.
	
	\textbf{Subcase 1.1: $xy<t$.} From the kernel estimate \eqref{eq-est1-proof Prop1} and the facts $1-r\sim t$ and $r\sim 1$,  we have
\[
\partial_t p_t^\nu(x,y)\lesi \f{1}{ t}\exp\Big(-\f{|x-y|}{ct}\Big)\Big(\f{t}{xy}\Big)^{\gamma_\nu}.
\]

\textbf{Subcase 1.2: $xy\ge t$.} We rewrite
	\[
\begin{aligned}
	\partial_t p^\nu_t(x,y)&= -\f{2(\nu+1)(1+r)}{1-r}p^\nu_t(x,y) +\f{4r}{(1-r)^2}(x-y)^2 p^\nu_t(x,y)\\
	& \ \ -\Big[\f{4r^{1/2}(1+r)}{(1-r)^2}-\f{8r}{(1-r)^2}\Big]xyp_t^{\nu}(x,y)-\f{4r^{1/2}(1+r)}{(1-r)^2}xy[p_t^{\nu+1}(x,y)-p_t^{\nu}(x,y)]\\
	&= -\f{2(\nu+1)(1+r)}{1-r}p^\nu_t(x,y) +\f{4r}{(1-r)^2}(x-y)^2 p^\nu_t(x,y)\\
	& \ \ -\f{4r^{1/2} (1-r^{1/2})^2}{(1-r)^2} xyp_t^{\nu}(x,y) -\f{4r^{1/2}(1+r)}{(1-r)^2}xy[p_t^{\nu+1}(x,y)-p_t^{\nu}(x,y)]\\
	&= -\f{2(\nu+1)(1+r)}{1-r}p^\nu_t(x,y) +\f{4r}{(1-r)^2}|x-y|^2 p^\nu_t(x,y)\\
	& \ \ -\f{4r^{1/2}}{(1+r^{1/2})^2} xyp_t^{\nu}(x,y)-\f{4r^{1/2}(1+r)}{(1-r)^2}xy[p_t^{\nu+1}(x,y)-p_t^{\nu}(x,y)]\\
	&=E_1 + E_2 + E_3 + E_4.
\end{aligned}
\]
Since in this situation $r\sim 1$ and $1-r\sim t$, we have
	\begin{equation*}
	\begin{aligned}
		|E_1| +|E_2| + |E_3|&\lesi \f{1}{t}p_t^\nu(x,y)+\f{|x-y|^2}{t^2}p_t^\nu(x,y) + xyp_t^\nu(x,y).
	\end{aligned}
\end{equation*}
This, together with the upper bound of $p_t^\nu(x,y)$ in \eqref{eq-est1-proof Prop1}, further yields
	\begin{equation*}
	\begin{aligned}
		|E_1| +|E_2| + |E_3|&\lesi  \f{1}{t^{3/2}}\exp\Big(-\f{|x-y|^2}{ct}\Big)\txg\tyg.
	\end{aligned}
\end{equation*}

For the last term $E_4$, using the facts $1-r \sim 
t$ and $r\sim 1$, \eqref{eq5-Inu} and Proposition \ref{prop-heat kernel} to obtain
\[
\begin{aligned}
	|E_4|&\lesi \f{xy}{t^2}|p_t^{\nu+1}(x,y)-p_t^{\nu+1}(x,y)|\\
	&\sim \f{xy}{t^2} \f{1-r}{2\sqrt r xy}p_t^{\nu+1}(x,y)\\
	&\sim \f{1}{t}p_t^{\nu}(x,y)\\
	&\lesi \f{1}{t^{3/2}}\exp\Big(-\f{|x-y|^2}{ct}\Big)\txg\tyg.
\end{aligned}
\]
This complete our proof for the case $t\in (0,1)$.

\noindent\textbf{Case 2: $t\ge 1$.}	In this case, $1+r\sim 1-r\sim 1$ and $r\le 1$. Hence, from \eqref{eq-time derivative of ptnu}, we have
\[
\begin{aligned}
	|\partial_t p_t^\nu(x,y)|&\lesi  p_t^\nu(x,y) +  (x^2+y^2)p_t^\nu(x,y) +xyp_t^{\nu+1}(x,y)\\
	&\lesi  p_t^\nu(x,y) +  (x^2+y^2)p_t^\nu(x,y) +xyp_t^{\nu}(x,y)\\
	&\lesi   p_t^\nu(x,y) +  (x^2+y^2)p_t^\nu(x,y),
\end{aligned}
\]
where in the last inequality we used $x^2+ y^2 \ge 2xy$.

\textbf{Subcase 2.1: $xy<e^{2t}$.} By \eqref{eq1-ptxy} and \eqref{eq1-Inu}    we have
\begin{equation*}
	\begin{aligned}
		p_t^\nu(x,y)&\lesi \f{1}{e^{t}}\exp\Big(-c(x^2+y^2)\Big)\Big(\f{1}{e^{-2t}xy}\Big)^{\gamma_\nu}\\
		&\lesi \f{1}{e^{t(1-2\gamma_\nu)}}\exp\Big(-c(x^2+y^2)\Big)\Big(\f{1}{xy}\Big)^{\gamma_\nu}.
	\end{aligned}
\end{equation*}
which implies
\[
p_t^\nu(x,y) +  (x^2+y^2)p_t^\nu(x,y)\lesi \f{1}{t^{3/2}}\exp\Big(-\f{|x-y|^2}{ct}\Big)\txg\tyg.
\]
as desired.

\medskip

\textbf{Subcase 2.2: $xy\ge e^{2t}$.}  By \eqref{eq2-ptxy}, \eqref{eq2-Inu} and \eqref{eq3-Inu}  we have
$$		p_t^\nu(x,y)\lesi \f{1}{e^t}\exp\Big(-c|x-y|^2\Big)\exp\Big(-c'xy\Big).
$$
If $y  < x/2$ or $y>2x$, then $|x-y|\sim \max\{x,y\}$. Hence, 
\[
\begin{aligned}
p_t^\nu(x,y) +  (x^2+y^2)p_t^\nu(x,y)\\
&\lesi  \f{1}{e^t}\exp\Big(-c|x-y|^2\Big)+(x^2+y^2)\f{t}{|x-y|^2}\f{1}{e^t}\exp\Big(-c|x-y|^2/2\Big)\\
&\lesi \f{1}{t^{3/2}}\exp\Big(-\f{|x-y|^2}{ct}\Big)\\&\lesi \f{1}{t^{3/2}}\exp\Big(-\f{|x-y|^2}{ct}\Big)\txg\tyg.	
\end{aligned}
\]

If $x/2\le y\le 2x$, then $x\sim y \gtrsim \sqrt t$. Hence, 
\begin{equation*}
	\begin{aligned}
		 p_t^\nu(x,y)  +  (x^2+y^2)p_t^\nu(x,y)&\sim p_t^\nu(x,y)  +  x^2p_t^\nu(x,y)\\
		&\lesi  \f{1}{e^t}\exp\Big(-c|x-y|^2\Big) + \f{x^2}{e^t}\exp\Big(-c|x-y|^2\Big)\f{1}{xy}\\
		&\lesi  \f{1}{e^t}\exp\Big(-c|x-y|^2\Big)\\
		&\lesi \f{1}{t^{3/2}}\exp\Big(-\f{|x-y|^2}{ct}\Big)\\&\lesi \f{1}{t^{3/2}}\exp\Big(-\f{|x-y|^2}{ct}\Big)\txg\tyg.	
	\end{aligned}
\end{equation*}

This completes our proof.
\end{proof}

The following result concerns estimates of the mixed derivative of the heat kernel $p_t^\nu(x,y)$.
\begin{prop}\label{prop- partial k heat kernel}
Let $\nu>-1$ and $a\ge 0$. For $k\in \mathbb N$, then we have
\begin{equation}\label{eq-partial k heat kernel 1}
| \partial^k_tp^{\nu}_t(x,y)|\lesi \f{1}{t^{k+1/2}}\exp\Big(-\f{|x-y|^2}{ct}\Big)\Big(1+\f{\sqrt t}{x} \Big)^{\gamma_\nu}\Big(1+\f{\sqrt t}{y} \Big)^{\gamma_\nu},
\end{equation}
\begin{equation}\label{eq-partial k heat kernel 2}
| \delta\partial^k_tp^{\nu}_t(x,y)|\lesi \f{1}{t^{k+1}}\exp\Big(-\f{|x-y|^2}{ct}\Big)\Big(1+\f{\sqrt t}{x} \Big)^{\gamma_\nu}\Big(1+\f{\sqrt t}{y} \Big)^{\gamma_\nu},
\end{equation}
and
\begin{equation}\label{eq-partial k heat kernel 3}
| \delta^*\partial^k_t[e^{-ta}p^{\nu+1}_t(x,y)]|\lesi \f{1}{t^{k+1}}\exp\Big(-\f{|x-y|^2}{ct}\Big)\Big(1+\f{\sqrt t}{x} \Big)^{\gamma_\nu}
\end{equation}
for all $t>0$ and $x,y\in (0,\vc)$.
\end{prop}
\begin{proof}
	We only provide the proof for $-1<\nu<-1/2$ since the proof corresponding to $\nu\ge -1/2$ is similar.
	
	For $a\in (0,1/2)$ and $c>0$, define
	\[
	H_{t,a,c}(x,y) = \f{1}{\sqrt t}\exp\Big(-\f{|x-y|^2}{ct}\Big)\Big(1+\f{\sqrt t}{x} \Big)^{\gamma_\nu}\Big(1+\f{\sqrt t}{y} \Big)^{\gamma_\nu}
	\]
	for $t>0$ and $x,y>0$.
	
	Then will show that 
	\begin{equation}\label{eq-product H less than H}
	\int_0^\vc  H_{t,a,c}(x,z)H_{t,a,c}(z,y)dz\lesi  H_{t,a,4c}(x,y)
	\end{equation}
	for $t>0$ and $x,y>0$.
	
	Indeed, we have
	\[
	\begin{aligned}
		\int_0^\vc  H_{t,a,c}(x,z)&H_{t,a,c}(z,y)dz\\
		=& \f{1}{t}\txg\tyg\\
		&\times \int_{(0,\vc}\exp\Big(-\f{|x-z|^2}{ct}\Big)\Big(1+\f{\sqrt t}{z}\Big)^{a}\exp\Big(-\f{|z-y|^2}{ct}\Big)\Big(1+\f{\sqrt t}{z}\Big)^{a} dz.
	\end{aligned}
	\]
	Using the following inequality
	\[
	|x-z|^2 + |z-y|^2\ge |x-y|^2/2,
	\]
	we have,
	\[
	\begin{aligned}
		\int_0^\vc  H_{t,a,c}(x,z)&H_{t,a,c}(z,y)dz\\
		\lesi&\f{1}{t}\exp\Big(-\f{|x-y|^2}{4ct}\Big)\txg\tyg\\
		&\times \int_{(0,\vc}\exp\Big(-\f{|x-z|^2}{2ct}\Big)\Big(1+\f{\sqrt t}{z}\Big)^{a}\exp\Big(-\f{|z-y|^2}{2ct}\Big)\Big(1+\f{\sqrt t}{z}\Big)^{a} dz\\
		&\lesi \f{1}{\sqrt t}H_{t,a,4c}(x,y) \Big[\int_0^\vc \exp\Big(-\f{|x-z|^2}{2ct}\Big)\Big(1+\f{\sqrt t}{z}\Big)^{2a}\Big]^{1/2}\\
		& \ \ \ \times \Big[\int_0^\vc \exp\Big(-\f{|y-z|^2}{2ct}\Big)\Big(1+\f{\sqrt t}{z}\Big)^{2a}\Big]^{1/2},
	\end{aligned}
	\]
	where in the last inequality we used the H\"older's inequality.
	
	On the other hand, by Lemma \ref{lem-elementary} we have
	\[
	\begin{aligned}
		\Big[\int_0^\vc \exp\Big(-\f{|x-z|^2}{ct}\Big)\Big(1+\f{\sqrt t}{z}\Big)^{2a}\Big]^{1/2}&\lesi \Big[\int_{0}^{\sqrt t} \Big(1+\f{\sqrt t}{z}\Big)^{2a}dz\Big]^{1/2}+\Big[\int_{0}^\vc \exp\Big(-\f{|x-z|^2}{ct}\Big)dz\Big]^{1/2}\\
		&\lesi t^{1/4},
	\end{aligned}
	\]
	as long as $a\in (0,1/2)$.
	
	By Lemma \ref{lem-elementary} again,
	\[
	\Big[\int_0^\vc \exp\Big(-\f{|x-z|^2}{ct}\Big)\Big(1+\f{\sqrt t}{z}\Big)^{2a}\Big]^{1/2}\lesi t^{1/4}.
	\]
	Consequently,
	\[
	\int_0^\vc  H_{t,a,c}(x,z)H_{t,a,c}(z,y)dz\lesi H_{t,a,4c}(x,y),
	\]
	which proves \eqref{eq-product H less than H}.
	
	We now turn to prove \eqref{eq-partial k heat kernel 1}. To do this, we write
	\[
	\begin{aligned}
		\partial_t^k e^{-t\mathcal L_\nu}=(-1)^k \mathcal L_\nu^ke^{-t\mathcal L_\nu}= k^k \big[-\tfrac{1}{k}\mathcal L_\nu e^{-\f{t}{k}\mathcal L_\nu}]^k=k^k \underbrace{\partial_t e^{-\f{t}{k}\mathcal L_\nu}\circ \ldots \circ \partial_t e^{-\f{t}{k}\mathcal L_\nu}}_{k \ \ \text{times}},
			\end{aligned}
	\] 
	which implies 
	\[
	\partial_t^k p_t^\nu(x,y) = k^k\int_{(0,\vc)^k}\partial_t^k p_{t/k}^\nu(x,z_1)\partial_t^k p_{t/k}^\nu(z_1,z_2)\ldots \partial_t^k p_{t/k}^\nu(z_{k-1},y)dz_1dz_2\ldots dz_{k-1}.
	\]
	Due to Proposition \ref{prop- time derivative of heat kernel},
	\[
	|\partial_t p_t^\nu(x,y)|\lesi H_{t,\gamma_\nu, c}(x,y), t>0, \ \ x,y>0
	\]
	for some $c$. 
	
	From this estimate and the fact that $H_{t,\gamma_\nu, c}(x,y)$ is in creasing in $c>0$, applying \eqref{eq-product H less than H} $k-1$ times, we come up with
	\[
	|\partial_t^k p_t^\nu(x,y)|\lesi H_{t/k,\gamma_\nu, 4^{k-1}c}(x,y), 
	\] 
	which ensures \eqref{eq-partial k heat kernel 1}.
	
	For \eqref{eq-partial k heat kernel 2}, we have
	\[
			\delta\partial_t^k e^{-t\mathcal L_\nu}= (k+1)^{k}\delta e^{-\f{t}{k+1}\mathcal L_\nu} \circ \underbrace{\partial_t e^{-\f{t}{k+1}\mathcal L_\nu}\circ \ldots \circ \partial_t e^{-\f{t}{k+1}\mathcal L_\nu}}_{k \ \ \text{times}}.
	\]
	Arguing similarly to the estimate of \eqref{eq-partial k heat kernel 1} by making use of Propositions \ref{prop- delta k pt} and \ref{prop- time derivative of heat kernel}, we obtain \eqref{eq-partial k heat kernel 2}.
	
	The proof of \eqref{eq-partial k heat kernel 3} is similar to that of \eqref{eq-partial k heat kernel 2}. Here we would like to bring an attention that since $p_t^{\nu+1}(x,y)$ satisfies the Gaussian upper bound due to \eqref{eq-ptxy nu > -1/2},  \cite[Lemma 2.5]{CD} tells us that $\partial_t p_t^{\nu+1}(x,y)$ satisfies the Gaussian upper bound. Of course, the Gaussian upper bound of $p_t^{\nu+1}(x,y)$ implies the similar upper bound as in Proposition \ref{prop- time derivative of heat kernel}.
	
	This completes our proof.
	
\end{proof}
\bigskip
\subsection{The case $n\ge 2$} 

\subsection{Kernel estimates}
Before coming to the estimate of the heat kernel $p_t^\nu(x,y)$ in $\mathbb R^n_+, n\ge 2$. We consider the following technical results.  Let $\sigma, \beta\in [0,1/2)$. For each $t>0$, define
\begin{equation}\label{eq-Tt}
	T_{t,\beta,\sigma}f(x)=\int_{(0,\vc)^n}T_{t,\beta,\sigma}(x,y)f(y)dy,
\end{equation}
where
\begin{equation}\label{eq-Tt kernel}
	T_{t,\beta,\sigma}(x,y)= \f{1}{t^{n/2}}\exp\Big(-\f{|x-y|^2}{ct}\Big)\prod_{j=1}^n\Big(1+\f{\sqrt t}{x_j}\Big)^\beta\Big(1+\f{\sqrt t}{y_j}\Big)^\sigma.
\end{equation}

\begin{lem}\label{lem Lpq}
	Let $\sigma\in (0,1/2)$. Let $T_{t,\beta,\sigma}$ be the linear operator defined by \eqref{eq-Tt} and \eqref{eq-Tt kernel}. Then for any $\f{1}{1-\sigma}<p\le q<\f{1}{\beta}$ and any ball $B\subset \mathbb R^n_+$, we have
	\begin{equation}
		\label{eq- eq1 Tt off diagonal}
		\|T_{t,\beta,\sigma}\|_{p\to q}\lesi t^{-\f{n}{2}(\f{1}{p}-\f{1}{q})},
	\end{equation}
	and
	\begin{equation}
		\label{eq- eq2 Tt off diagonal}
		\|T_{t,\beta,\sigma}\|_{L^p(B)\to L^q(S_j(B))}+\|T_{t,\beta,\sigma}\|_{L^p(S_j(B))\to L^q(B)}\lesi t^{-\f{n}{2}(\f{1}{p}-\f{1}{q})}\exp\Big(-\f{(2^jr_B)^2}{ct}\Big), j\ge 2.
	\end{equation}
\end{lem}
\begin{proof}
	Fix $\f{1}{1-\sigma}<p\le q<\f{1}{\beta}$. Let $f\in L^p(\mathbb R^n_+)$. We first prove \eqref{eq- eq1 Tt off diagonal} for $n=1$. Indeed, in this case we have
	\[
	\begin{aligned}
		\|T_{t,\beta,\sigma}f\|_{L^q(B)}&\lesi \Big[\int_{\sqrt t}^\vc\Big(\int_{\sqrt t}^\vc|T_{t,\beta,\sigma}(x,y)| |f(y)|dy\Big)^qdx\Big]^{1/q}\\
		& \ \ +\Big[\int_{\sqrt t}^\vc\Big(\int^{\sqrt t}_0|T_{t,\beta,\sigma}(x,y)| |f(y)|dy\Big)^qdx\Big]^{1/q}\\
		& \ \ +\Big[\int^{\sqrt t}_0\Big(\int_{\sqrt t}^\vc|T_{t,\beta,\sigma}(x,y)| |f(y)|dy\Big)^qdx\Big]^{1/q}\\
		& \ \ +\Big[\int^{\sqrt t}_0\Big(\int^{\sqrt t}_0|T_{t,\beta,\sigma}(x,y)| |f(y)|dy\Big)^qdx\Big]^{1/q}=:E_1+E_2+E_3+E_4.
	\end{aligned}
	\]
	Since for $x\ge \sqrt t$ and $y\ge \sqrt t$, the kernel $T_{t,\beta,\sigma}(x,y)$ satisfies the Gaussian upper bound, i.e.,
	\[
	|T_{t,\beta,\sigma}(x,y)|\lesi \f{1}{t^{n/2}}\exp\Big(-\f{|x-y|^2}{ct}\Big),
	\]
	as a well-known result,
	\[
	E_1\lesi t^{-\f{n}{2}(\f{1}{p}-\f{1}{q})}.
	\]
	For the term $E_2$, by the Minkowski inequality we have
	\[
	\begin{aligned}
		E_2 &\lesi \Big[\int_{\sqrt t}^\vc\Big(\int^{\sqrt t}_0\f{1}{\sqrt t}\exp\Big(-\f{|x-y|^2}{ct}\Big) \Big(1+\f{\sqrt t}{x}\Big)^\beta\Big(1+\f{\sqrt t}{y}\Big)^\sigma |f(y)|dy\Big)^qdx\Big]^{1/q}\\
		&\sim \Big[\int_{\sqrt t}^\vc\Big(\int^{\sqrt t}_0\f{1}{\sqrt t}\exp\Big(-\f{|x-y|^2}{ct}\Big)  \Big(1+\f{\sqrt t}{y}\Big)^\sigma |f(y)|dy\Big)^qdx\Big]^{1/q}\\
		&\lesi \f{1}{\sqrt t}\Big[\int_{\sqrt t}^\vc  \exp\Big(-\f{q|x-y|^2}{ct}\Big) dx\Big]^{1/q}\times \int^{\sqrt t}_0 \Big(1+\f{\sqrt t}{y}\Big)^\sigma |f(y)|dy.
	\end{aligned}
	\]
	By Lemma \ref{lem-elementary},
	\[
	\int_{\sqrt t}^\vc  \exp\Big(-\f{q|x-y|^2}{ct}\Big) dx \lesi \sqrt t.
	\]
	On the other hand, by the H\"older inequality, Lemma \ref{lem-elementary}  and the fact that $\f{1}{1-\sigma}<p, p'<\f{1}{\sigma}$, we have
	\[
	\begin{aligned}
		\int^{\sqrt t}_0 \Big(1+\f{\sqrt t}{y}\Big)^\sigma |f(y)|dy&\lesi \|f\|_p \Big[\int^{\sqrt t}_0 \Big(1+\f{\sqrt t}{y}\Big)^{\sigma p'}dy\Big]^{1/p'}\\
		&  \lesi t^{\f{1}{2p'}}\|f\|_p.
	\end{aligned}
	\]
	Hence,
	\[
	E_2\lesi t^{-\f{1}{2}+\f{1}{2q}+\f{1}{2p'}}\|f\|_p=t^{-\f{1}{2}(\f{1}{p}-\f{1}{q})}\|f\|_p.
	\]
	Similarly,
	\[
	\begin{aligned}
		E_3 &\lesi  \f{1}{\sqrt t}\Big[\int^{\sqrt t}_0\Big(1+\f{\sqrt t}{x}\Big)^{\beta q}\Big(\int_{\sqrt t}^\vc\exp\Big(-\f{|x-y|^2}{ct}\Big)   |f(y)|dy\Big)^qdx\Big]^{1/q}.
	\end{aligned}
	\]
	By the H\"older inequality and Lemma \ref{lem-elementary} ,
	\[
	\begin{aligned}
		\int_{\sqrt t}^\vc\exp\Big(-\f{|x-y|^2}{ct}\Big)   |f(y)|dy&\lesi \|f\|_p\Big(\int_{\sqrt t}^\vc\exp\Big(-\f{p'|x-y|^2}{ct}\Big) dy\Big)^{1/p'}\\
		&\lesi t^{\f{1}{2p'}}\|f\|_p.
	\end{aligned}
	\]
	Hence,
	\[
	\begin{aligned}
		E_3&\lesi t^{\f{1}{2p'}-\f{1}{2}}\|f\|_p\Big[\int^{\sqrt t}_0\Big(1+\f{\sqrt t}{x}\Big)^{\beta q} dx\Big]^{1/q}\\
		&\lesi t^{\f{1}{2p'}-\f{1}{2}+\f{1}{2q}}\|f\|_p=t^{-\f{1}{2}(\f{1}{p}-\f{1}{q})}\|f\|_p,
	\end{aligned}
	\]
	as long as $q<1/\beta.$
	
	For the last term, by the H\"older inequality and Lemma \ref{lem-elementary} ,
	\[
	\begin{aligned}
		E_3 &\lesi  \f{1}{\sqrt t}\Big[\int^{\sqrt t}_0\Big(1+\f{\sqrt t}{x}\Big)^{\beta q}dx\Big]^{1/q}\Big[\int^{\sqrt t}_0\Big(1+\f{\sqrt t}{y}\Big)^{\sigma }  |f(y)|dy\Big]^q\\
		&\lesi  \f{1}{\sqrt t}\Big[\int^{\sqrt t}_0\Big(1+\f{\sqrt t}{x}\Big)^{\beta q}dx\Big]^{1/q}\Big[\int^{\sqrt t}_0\Big(1+\f{\sqrt t}{y}\Big)^{\sigma p'}dy\Big]^{1/p'}\|f\|_p\\
		&\lesi t^{-\f{1}{2}+\f{1}{2q}+\f{1}{2p'}}\|f\|_p=t^{-\f{1}{2}(\f{1}{p}-\f{1}{q})}\|f\|_p.
	\end{aligned}
	\]
	This completes the proof of \eqref{eq- eq1 Tt off diagonal} for $n=1$.
	
	We now prove for $n\ge 2$. For each $j=1,\ldots, n$, define
	\[
	T_{t,\beta,\sigma}^j f(x_j) = \int_0^\vc T^j_{t,\beta,\sigma}(x_j,y_j)f(y_j)dy_j,
	\]
	where
	\[
	T^j_{t,\beta,\sigma}(x,y)= \f{1}{\sqrt t}\exp\Big(-\f{|x_j-y_j|^2}{ct}\Big)\Big(1+\f{\sqrt t}{x_j}\Big)^\beta\Big(1+\f{\sqrt t}{y_j}\Big)^\sigma.
	\]	
	Then for $f\in L^p(\mathbb R^n_+)$, we can write
	\[
	T_{t,\beta,\sigma}f(x) = T^1_{t,\beta,\sigma} \otimes \ldots \otimes T^n_{t,\beta,\sigma}f(x_1,\ldots, x_n). 
	\]
	At this stage, applying 	\eqref{eq- eq1 Tt off diagonal} (for $n=1$) $n$-times we will come up with
	\[
	\|T_{t,\beta,\sigma}f\|_{p\to q} \le \prod_{j=1}^n\|T^j_{t,\beta,\sigma}\|_{p\to q}\lesi t^{-\f{n}{2}(\f{1}{p}-\f{1}{q})}.
	\] 
	
	This completes the proof of \eqref{eq- eq1 Tt off diagonal}.
	
	\bigskip
	
	We now take care of \eqref{eq- eq2 Tt off diagonal}. We only prove 
	\[
	\|T_{t,\beta,\sigma}\|_{L^p(B)\to L^q(S_j(B))} \lesi t^{-\f{n}{2}(\f{1}{p}-\f{1}{q})}\exp\Big(-\f{(2^jr_B)^2}{ct}\Big), j\ge 2,
	\]
	since the estimate of $\|T_{t,\beta,\sigma}\|_{L^p(S_j(B))\to L^q(B)}$ can be done similarly.

	To do this, let $B$ be a ball in $\mathbb R^n_+$ and $f\in L^p(B)$. For $j\ge 2$, fix $x^0 \in S_j(B)$ and $y^0\in B$ then for $x\in S_j(B)$, $y\in B$  and $c>0$ we have
	\[
	|x-y|\sim |x-y^0|\sim |y-x^0|\sim 2^jr_B,
	\]
	which yields
	\[
	\exp\Big(-\f{|x-y|^2}{ct}\Big) \lesi \exp\Big(-\f{|x-y^0|^2}{c't}\Big)\exp\Big(-\f{(2^jr_B)^2}{c't}\Big) \exp\Big(-\f{|y-x^0|^2}{c't}\Big), t>0
	\]
	for some $c'>0$.
	
	Using this inequality, we have
	\[
	\begin{aligned}
		\|T_{t,\beta,\sigma}f\|_{L^q(S_j(B))} &\lesi \Big[\int_{S_j(B)} \Big(\int_B\f{1}{t^{n/2}}\exp\Big(-\f{|x-y|^2}{ct}\Big)\prod_{j=1}^n\Big(1+\f{\sqrt t}{x_j}\Big)^\beta\Big(1+\f{\sqrt t}{y_j}\Big)^\sigma |f(y)| dy\Big)^q dx\Big]^{1/q}\\
		&\lesi t^{-n/2}\exp\Big(-\f{(2^jr_B)^2}{c't}\Big) \Big[\int_{S_j(B)}\exp\Big(-\f{q|x-y^0|^2}{c't}\Big)\prod_{j=1}^n\Big(1+\f{\sqrt t}{x_j}\Big)^{\beta q} dx\Big]^{1/q}\\
		& \ \ \times \int_B \exp\Big(-\f{|y-x^0|^2}{c't}\Big)\prod_{j=1}^n\Big(1+\f{\sqrt t}{y_j}\Big)^{\sigma }|f(y)|dy.
	\end{aligned}
	\]
	From Lemma \ref{lem-elementary}, we imply the following inequality
	\begin{equation}\label{eq1}
		\int_0^\vc \exp\Big(\f{|x_j-z_j|^2}{ct}\Big)\Big(1+\f{\sqrt t}{x_j}\Big)^{\beta q}dx_j \lesi \sqrt t, 
	\end{equation}
	uniformly in $z_j>0$. Applying this inequality $n$-times, 
	\[
	\Big[\int_{S_j(B)}\exp\Big(-\f{q|x-y^0|^2}{c't}\Big)\prod_{j=1}^n\Big(1+\f{\sqrt t}{x_j}\Big)^{\beta q} dx\Big]^{1/q}\lesi t^{\f{n}{2q}}.
	\]
	Hence,
	\[
	\begin{aligned}
		\|T_{t,\beta,\sigma}f\|_{L^q(S_j(B))} &\lesi t^{-\f{n}{2}+\f{n}{2q}} \exp\Big(-\f{(2^jr_B)^2}{c't}\Big)\int_B \exp\Big(-\f{q|y-x^0|^2}{c't}\Big)\prod_{j=1}^n\Big(1+\f{\sqrt t}{y_j}\Big)^{\sigma }|f(y)|dy\\
		&\lesi t^{-\f{n}{2}+\f{n}{2q}} \exp\Big(-\f{(2^jr_B)^2}{c't}\Big)\Big[\int_B \exp\Big(-\f{p'|y-x^0|^2}{c't}\Big)\prod_{j=1}^n\Big(1+\f{\sqrt t}{y_j}\Big)^{\sigma p'} dy\Big]^{1/p'}\|f\|_p\\
		&\lesi t^{-\f{n}{2}+\f{n}{2q}+\f{n}{2p'}} \exp\Big(-\f{(2^jr_B)^2}{c't}\Big)\|f\|_p\\
		&\lesi t^{-\f{n}{2}(\f{1}{p}-\f{1}{q})} \exp\Big(-\f{(2^jr_B)^2}{c't}\Big)\|f\|_p,
	\end{aligned}
	\]
	where in the second inequality we used the H\"older inequality and in the third inequality we applied \eqref{eq1} $n$-times.
	
	This completes our proof. 
\end{proof}

We now turn to some estimates regarding the heat kernel $p_t^\nu(x,y)$. Let $\nu\in (-1,\vc)^n$. Recall that 
\begin{equation*} 
	\gamma_\nu = \begin{cases}
		\max\{-1/2-\nu_j: j\in \mathcal I_\nu\}, \ \ & \ \text{if} \  \mathcal I_\nu\ne \emptyset, \\
		0, \ \ &\ \text{if} \  \mathcal I_\nu= \emptyset,
	\end{cases}
\end{equation*}
where $\mathcal I_\nu=\{j: -1<\nu_j<-1/2\}$.

For each $j$, define
\begin{equation*} 
	\gamma_{\nu_j} = \begin{cases}
		-1/2-\nu_j, & \ -1<\nu_j<-1/2, \\
		0, \ \ &\ \nu_j\ge -1/2.
	\end{cases}
\end{equation*}
Then we have
\begin{equation}\label{eq-gamma max}
\gamma_\nu=\max\{\gamma_{\nu_1}, \ldots, \gamma_{\nu_n}\}.
\end{equation}

From \eqref{eq-gamma max}, Proposition \ref{prop- partial k heat kernel} and Lemma \ref{lem Lpq}, we have: 
\begin{prop}\label{prop-  heat kernel n ge 1}
	Let $\nu\in (-1,\vc)^n$ and $a\ge 0$. For $k\in \mathbb N$, then we have
	\begin{equation}\label{eq-partial k heat kernel 1 higher dimension}
		| \partial^k_tp^{\nu}_t(x,y)|\lesi \f{1}{t^{k+n/2}}\exp\Big(-\f{|x-y|^2}{ct}\Big)\prod_{j=1}^n\Big(1+\f{\sqrt t}{x_j} \Big)^{\gamma_{\nu_j}}\Big(1+\f{\sqrt t}{y_j} \Big)^{\gamma_{\nu_j}},
	\end{equation}
	\begin{equation}\label{eq-partial k heat kernel 2 higher dimension}
		| \delta\partial^k_tp^{\nu}_t(x,y)|\lesi \f{1}{t^{k+(n+1)/2}}\exp\Big(-\f{|x-y|^2}{ct}\Big) \prod_{j=1}^n\Big(1+\f{\sqrt t}{x_j} \Big)^{\gamma_{\nu_j}}\Big(1+\f{\sqrt t}{y_j} \Big)^{\gamma_{\nu_j}},
	\end{equation}
	and for $\ell=1,\ldots, n$,
	\begin{equation}\label{eq-partial k heat kernel 3 higher dimension}
		| \delta_\ell^*\partial^k_t[e^{-at}p^{\nu+e_\ell}_t(x,y)]|\lesi \f{1}{t^{k+(n+1)/2}}\exp\Big(-\f{|x-y|^2}{ct}\Big)\Big(1+\f{\sqrt t}{x_\ell} \Big)^{\gamma_{\nu_\ell}}\prod_{j\ne \ell }\Big(1+\f{\sqrt t}{x_j} \Big)^{\gamma_{\nu_j}}\Big(1+\f{\sqrt t}{y_j} \Big)^{\gamma_{\nu_j}}
	\end{equation}
	for all $t>0$ and $x,y\in \mathbb R^n_+$.
	\end{prop}

\section{Singular integrals in Laguerre settings}

\subsection{The maximal function}
\begin{proof}[Proof of Theorem \ref{main thm 0}:] If $\nu\ge -1/2$, then by Proposition \ref{prop-heat kernel} the heat kernel $p_t^\nu(x,y)$ satisfies the Gaussian upper bound. Hence, by Lemma \ref{lem-elementary}, we have
	\[
	\mathcal M_{\mathcal L_\nu}f\lesi \mathcal M f.
	\]
The theorem follows from the above in equality and Lemma \ref{Lem-maximalfunction}.

We now consider the case $-1<\nu<-1/2$. Fix $\f{1}{1-\gamma_\nu}:=(\f{1}{\gamma_\nu})'<p<\f{1}{\gamma_\nu}$ and $w\in A_{p(1-\gamma_\nu)}\cap RH_{(\f{1}{p\gamma_\nu})'}$. From Lemma \ref{lem 1}, we have, for $x\in \mathbb R^n_+$,
	\[
	\begin{aligned}
		\sup_{t>0}|e^{-t\mathcal L_\nu}f(x)|&\lesi \sup_{t>0}\int_0^\vc\f{1}{\sqrt t}\exp\Big(-\f{|x-y|^2}{ct}\Big)|f(y)|dy\\
		&\ \ + \int_0^{x/2} \f{1}{x}\Big(\f{x}{y}\Big)^{\gamma_\nu} |f(y)| dy +\int_{x/2}^{\vc} \f{1}{y}\Big(\f{y}{x}\Big)^{\gamma_\nu} f(y) dy\\
		&=: T_1f(x) +T_2f(x) + T_3f(x),
	\end{aligned}
	\]
	which implies
	\[
	\Big\|\sup_{t>0}|e^{-t\mathcal L_\nu}f(x)|\Big\|_{L^p_w(\mathbb R_+)}\le \|T_1f\|_{L^p_w(\mathbb R_+)} +\|T_2f\|_{L^p_w(\mathbb R_+)} + \|T_3f\|_{L^p_w(\mathbb R_+)}.
	\]
	By Lemma \ref{lem-elementary}, we have
	\[
	T_1f(x)\lesi \mathcal Mf(x), x\in \mathbb R^n,
	\]
	which implies
	\[
	\|T_1f\|_{L^p_w(\mathbb R_+)}\lesi \|f\|_{L^p_w(\mathbb R_+)},
	\]
	since $w\in w\in A_{p(1-\gamma_\nu)}\cap RH_{(\f{1}{p\gamma_\nu})'} \subset A_p(\mathbb R_+)$.
	
	In order to estimate $T_2f$ and $T_3f$, we note that due to Lemma \ref{weightedlemma1} (vii), $w^{1-p'}\in A_{p'(1-\gamma_\nu)}$. This, together with the fact $w\in A_{p(1-\gamma_\nu)}$ and Lemma \ref{weightedlemma1} (iv), implies that there exists $ \min\{p,p'\}>r>\f{1}{1-\gamma_\nu}= (\f{1}{\gamma_\nu})'$, which implies $r'<\f{1}{\gamma_\nu}$, such that $w\in A_{p/r}$ and $w^{1-p'}\in A_{p'/r}$.

	We now turn to take care of $T_2f$ and $T_3f$. By  the H\"older inequality and the fact $\gamma_\nu r'<1$,
	\begin{equation}\label{eq-Tnu}
		\begin{aligned}
			|T_2 f(x)|&\le \int_0^{2x} \f{1}{x}\Big(\f{x}{y}\Big)^{\gamma_\nu}|f(y)|dy\\
			&\le \Big[\f{1}{x}\int_0^{2x}  |f(y)|^{r}dy\Big]^{1/r}\Big[\f{1}{x}\int_0^{2x}\Big(\f{x}{y}\Big)^{\gamma_\nu r'}   dy\Big]^{1/r'_1} \\
			&\lesi \Big[\f{1}{x}\int_0^{2x}  |f(y)|^{r}dy\Big]^{1/r}\\
			&\lesi \mathcal M_{r}f(x).
		\end{aligned}
	\end{equation}
	This, in combination with Lemma \ref{Lem-maximalfunction}, yields
	\[
	\|T_2f\|_{L^p_w(\mathbb R_+)}\lesi \|f\|_{L^p_w(\mathbb R_+)}
	\]
	since $w\in A_{p/r}$.
	
	The term $T_3f$ can be done by duality. Let $g\in L^{p'}_{w^{1-p'}}(\mathbb R^n_+)$. Then we have
	\[
	\begin{aligned}
		\langle T_3 f,g \rangle &=\int_0^\vc\int_{x/2}^{\vc}g(x) \f{1}{y}\Big(\f{y}{x}\Big)^{\gamma_\nu} f(y) dy dx\\
		&\lesi \int_0^\vc|f(y)|\int_{0}^{2y} \f{1}{y}\Big(\f{y}{x}\Big)^{\gamma_\nu}|g(x)|  dx dy
	\end{aligned}
	\]
	Similarly to \eqref{eq-Tnu}, since $\gamma_\nu r'<1$ we have
	\[
	\int_{0}^{2y} \f{1}{y}\Big(\f{y}{x}\Big)^{\gamma_\nu}|g(x)|  dx\lesi \mathcal M_{r}g(y),
	\]	
	which implies
	\[
	\begin{aligned}
		\langle T_3 f,g \rangle &\lesi \langle \mathcal M_{r}g(y),f \rangle\\
		& \lesi \|\mathcal M_{r}g\|_{L^{p'}_{w^{1-p'}}(\mathbb R^n_+)}\|f\|_{L^{p}_{w}(\mathbb R^n_+)}\\
		&\lesi \|g\|_{L^{p'}_{w^{1-p'}}(\mathbb R^n_+)}\|f\|_{L^{p}_{w}(\mathbb R^n_+)},
	\end{aligned}
	\]
	where in the last inequality we used Lemma \ref{Lem-maximalfunction} since $w^{1-p'}\in A_{p'/r}$.
	
	It follows that
	\[
	\|T_3f\|_{L^p_w(\mathbb R_+)}\lesi \|f\|_{L^p_w(\mathbb R_+)}.
	\]
	Therefore, the proof is completed.
\end{proof}

\subsection{The Riesz transform}

\begin{proof}[Proof of Theorem \ref{main thm 1}:]
	
	We first prove that $R_\nu$ is bounded on $\frac{1}{1-\gamma_\nu}<p\le 2$.   Fix $m\in \mathbb N$ such that $2m-1>n$. In the light of Theorem \ref{BZ-thm}, it suffices to prove that for any $\f{1}{1-\gamma_\nu}<p_0<2$  
	\begin{equation}\label{eq1-proof of mainthm}
		\begin{aligned}
			\Big(\fint_{S_j(B)}|R_\nu(I-e^{-r_B^2\mathcal L_\nu})^mf|^{2} dx \Big)^{1/2}\lesi 2^{-j(2m-1)}\Big(\fint_B |f|^{2}dx\Big)^{1/2}
		\end{aligned}
	\end{equation}
	and
		\begin{equation}\label{eq2-proof of mainthm}
		\begin{aligned}
			\Big(\fint_{S_j(B)}|[I-(I-e^{-r_B^2\mathcal L_\nu})^m]f|^{2}dx\Big)^{1/2}\leq
	2^{-j(n+1)}\Big(\fint_B |f|^{p_0}dx\Big)^{1/p_0} 
\end{aligned}
\end{equation}
	for all $j\ge 2$ and $f\in L^p(\mathbb R^n_+)$ supported in $B$.

Note that	\[
	I-(I-e^{-r_B^2\mathcal L_\nu})^m =\sum_{k=1}^m c_ke^{-kr_B^2\mathcal L_\nu}, 
	\]
	where $c_k$ are constants.
	
	This, along with Proposition \ref{prop-  heat kernel n ge 1}, implies \eqref{eq2-proof of mainthm}. It remains to prove \eqref{eq1-proof of mainthm}. To do this, recall that
	\[
	R_\nu = c\int_0^\vc \sqrt t \delta e^{-t\mathcal L_\nu} \f{dt}{t}
	\]
	and
	\[
	(I-e^{-r_B^2\mathcal L_\nu})^m=\int_{[0,r_B^2]^m}\mathcal L_\nu^m e^{-(s_1+\ldots+s_m)\mathcal L_\nu}d\vec{s},
	\]
	where $d\vec{s}=ds_1\ldots ds_m$.
	
	Hence, we can write
	\[
	R_\nu(I-e^{-r_B^2\mathcal L_\nu})^mf =c\int_0^\vc\int_{[0,r_B^2]^m} \sqrt t \delta \mathcal L_\nu^{m}e^{-(t+s_1+\ldots+s_m)\mathcal L_\nu} fds\f{dt}{t}. 
	\]
	Applying  Proposition \ref{prop-  heat kernel n ge 1} and Lemma \ref{lem Lpq}, we obtain
	\[
	\begin{aligned}
		\|R_\nu&(I-e^{-r_B^2\mathcal L_\nu})^mf\|_{L^{2}(S_j(B))}\\&\lesi \int_0^\vc\int_{[0,r_B^2]^m} \sqrt t \big\|\delta \mathcal L_\nu^{m}e^{-(t+s_1+\ldots+s_m)\mathcal L_\nu} f\big\|_{L^{2}(S_j(B))}ds\f{dt}{t}\\
		&\lesi \|f\|_{L^{2}(B)}\int_0^\vc\int_{[0,r_B^2]^m} \f{\sqrt t}{(t+|\vec s|)^{m+\f{1}{2}}}\exp\Big(-\f{(2^jr_B)^2}{c(t+|\vec s|)}\Big)  d\vec{s}\f{dt}{t}\\
		&\lesi \|f\|_{L^{2}(B)}\int_0^{r_B^2}\ldots \f{dt}{t}+\|f\|_{L^{2}(B)}\int_{r_B^2}^\vc\ldots \f{dt}{t}\\
		&=: E_1 + E_2,
	\end{aligned}
	\]
	where $|\vec s| =s_1+\ldots+s_m$.
	
	For the first term $E_1$, we have
	\[
	\begin{aligned}
		E_1&\lesi \|f\|_{L^{2}(B)}\int_0^{r_B^2}\int_{[0,r_B^2]^m} \f{\sqrt t}{(t+|\vec s|)^{m +\f{1}{2}}} \Big( \f{t+|\vec s|}{(2^jr_B)^2}\Big)^{m  +\f{1}{2}}  d\vec{s}\f{dt}{t}\\
		 &\lesi \|f\|_{L^{2}(B)}\int_0^{r_B^2}\int_{[0,r_B^2]^m}  \f{\sqrt t}{(2^jr_B)^{2m+1}}   d\vec{s}\f{dt}{t}\\
		 &\sim \|f\|_{L^{2}(B)} \f{r_B^{2m+1}}{(2^jr_B)^{2m+1}}\\
		 &\sim  2^{-j(2m+1)}\|f\|_{L^{2}(B)} .
	\end{aligned}
	\]
	For the second term, note that $t+s_1+\ldots + s_m \sim t $ as $t\ge r_B^2$ and $s_1,\ldots, s_m\in [0,r_B^2]$. It follows that 
	\[
	\begin{aligned}
		E_2&\lesi \|f\|_{L^{2}(B)}\int_{r_B^2}^\vc\int_{[0,r_B^2]^m} \f{\sqrt t}{t^{m +\f{1}{2}}} \Big( \f{t}{(2^jr_B)^2}\Big)^{m  -\f{1}{2}}  d\vec{s}\f{dt}{t}\\
		&\sim 2^{-j(2m-1)}\|f\|_{L^{2}(B)}.
	\end{aligned}
	\]
	Collecting the estimates of $E_1$ and $E_2$, we deduce to \eqref{eq1-proof of mainthm}.

	\bigskip
	
Next we need to prove that $R_\nu^j$ is bounded on $L^p$ for $2< p < \f{1}{\gamma_{\nu+e_j}}$. By duality, it suffices to prove $(R^j_\nu)^*:=\mathcal L_\nu^{-1/2}\delta_j^*$ is bounded on $L^p(\mathbb R^n_+)$ for $\f{1}{1-\gamma_{\nu+e_j}}<p_0\le 2$. To do this, fix $\f{1}{1-\gamma_{\nu+e_j}}<p_0\le 2$, $m\in \mathbb N$ such that $2m-1>n$, by Theorem \ref{BZ-thm}, we need to show that  
	\begin{equation}\label{eq1-proof of mainthm dual Riesz transform}
		\begin{aligned}
			\Big(\fint_{S_j(B)}|(R^j_\nu)^*(I-e^{-r_B^2(\mathcal L_{\nu+e_j}+2)})^mf|^2 dx \Big)^{1/2}\lesi 2^{-j(2m-1)}\Big(\fint_B |f|^2dx\Big)^{1/2}
		\end{aligned}
	\end{equation}
	and
			\begin{equation}\label{eq2-proof of mainthm dual}
		\begin{aligned}
			\Big(\fint_{S_j(B)}|[I-(I-e^{-r_B^2(\mathcal L_{\nu+e_j}+2)})^m]f|^{2}dx\Big)^{1/2}\leq
			2^{-j(n+1)}\Big(\fint_B |f|^{p_0}dx\Big)^{1/p_0} 
		\end{aligned}
	\end{equation}
	for all $j\ge 2$ and $f\in L^2(\mathbb R^n_+)$ supported in $B$.
	
	The estimate \eqref{eq2-proof of mainthm dual} follows directly from Proposition \ref{prop-  heat kernel n ge 1} and Lemma \ref{lem Lpq} and the following expansion
	\[
	I-(I-e^{-r_B^2(\mathcal L_{\nu+e_j}+2)})^m=\sum_{k=1}^mc_k e^{-kr_B^2(\mathcal L_{\nu+e_j}+2)},
	\]
	where $c_k$ are constants.
	
	Hence, we need to prove \eqref{eq1-proof of mainthm dual Riesz transform}. We first note that for $f\in L^2(\mathbb R^n_+)$ and $\nu \in (-1,\vc)^n$ we have, for $t>0$,
	\begin{equation*}
		\mathcal L_{\nu}^{-1/2} e^{-t\mathcal L_{\nu} }f =\sum_{k\in \mathbb N^n} \f{e^{-t(4|k|+2|\nu|+2n)}}{\sqrt{4|k|+2|\nu|+2n}}\langle f,\varphi^\nu_k\rangle \varphi^\nu_k.
	\end{equation*}
	Using this and \eqref{eq- delta and eigenvector}, by a simple calculation we come up with
	
	\[
	e^{-t\mathcal L_{\nu}} \mathcal L_\nu^{-1/2}\delta_j^*(I-e^{-r_B^2(\mathcal L_{\nu+e_j}+2}))^mf= \delta^*_j e^{-t (\mathcal L_{\nu+e_j}+2)}(\mathcal L_{\nu+e_j}+2)^{-1/2}(I-e^{-r_B^2(\mathcal L_{\nu+e_j}+2}))^mf.	
	\]
	At this stage, arguing similarly to the proof of \eqref{eq1-proof of mainthm} we come up with \eqref{eq1-proof of mainthm dual Riesz transform}.
	
	We now turn to the weighted estimate for the Riesz transform $R^j_\nu$. By duality it suffices to show that $(R^j_\nu)^*$ is bounded on $L^p_w(\mathbb R^n_+)$ for all $\f{1}{1-\gamma_{\nu+e_j}}<p<\f{1}{\gamma_\nu}$ and $w\in A_{\f{p}{1-\gamma_{\nu+e_j}}}\cap RH_{(\f{1}{p\gamma_\nu})'}$. We now fix $\f{1}{1-\gamma_{\nu+e_j}}<p<\f{1}{\gamma_\nu}$ and $w\in A_{\f{p}{1-\gamma_{\nu+e_j}}}\cap RH_{(\f{1}{p\gamma_\nu})'}$. By Lemma \ref{weightedlemma1}, we can find $\f{1}{1-\gamma_{\nu+e_j}}<p_0<p<q_0<\f{1}{\gamma_\nu}$ such that $w\in A_{\f{p_0}{1-\gamma_{\nu+e_j}}}\cap RH_{(\f{1}{q_0\gamma_\nu})'}$. Similarly to \eqref{eq1-proof of mainthm dual Riesz transform} and \eqref{eq2-proof of mainthm dual} we also obtain
	\begin{equation*}
		\begin{aligned}
			\Big(\fint_{S_j(B)}|(R^j_\nu)^*(I-e^{-r_B^2(\mathcal L_{\nu+e_j}+2)})^mf|^2 dx \Big)^{1/q_0}\lesi 2^{-j(2m-1)}\Big(\fint_B |f|^2dx\Big)^{1/q_0}
		\end{aligned}
	\end{equation*}
	and
	\begin{equation*}
		\begin{aligned}
			\Big(\fint_{S_j(B)}|[I-(I-e^{-r_B^2(\mathcal L_{\nu+e_j}+2)})^m]f|^{2}dx\Big)^{1/2}\leq
			2^{-j(n+1)}\Big(\fint_B |f|^{p_0}dx\Big)^{1/p_0} 
		\end{aligned}
	\end{equation*}
	for all $j\ge 2$ and $f\in L^2(\mathbb R^n_+)$ supported in $B$.
	
	Hence, by Theorem \ref{BZ-thm}, $(R^j_\nu)^*$ is bounded on $L^{p}_w(\mathbb R^n_+)$ with $w\in A_{\f{p_0}{1-\gamma_{\nu+e_j}}}\cap RH_{(\f{1}{q_0\gamma_\nu})'}$.
	 
	This completes our proof. 
\end{proof}

\subsection{The square functions}

\begin{proof}[Proof of Theorem \ref{main thm 2}:]
	We only give the proof for the square function $S^j_{\mathcal L_\nu}$ since the proof for the square function $G_{\mathcal L_\nu}$ is similar and even easier.
	
	 From the $L^2(\mathbb R^n_+)$ boundedness of the Riesz transform $\delta \mathcal L_\nu^{-1/2}$, we have
	\[
	\begin{aligned}
		\|S^j_{\mathcal L_\nu}f\|_{2} &=\Big(\int_0^\vc \|\sqrt t \delta_j e^{-t\mathcal L_\nu}f\|_2^2\f{dt}{t} \Big)^{1/2}\\
		&\lesi \Big(\int_0^\vc \|\sqrt t \mathcal L_\nu^{1/2} e^{-t\mathcal L_\nu}f\|_2^2\f{dt}{t} \Big)^{1/2}\\
		&\lesi \Big\|\Big(\int_0^\vc |\sqrt t \mathcal L_\nu^{1/2} e^{-t\mathcal L_\nu}f|^2\f{dt}{t} \Big)^{1/2}\Big\|_2\\
		&\lesi \|f\|_2,
	\end{aligned}
	\]
	where in the last inequality we used the $L^2$-boundedness of the following square function 
	\[
	f\mapsto \Big(\int_0^\vc |\sqrt t \mathcal L_\nu^{1/2} e^{-t\mathcal L_\nu}f|^2\f{dt}{t} \Big)^{1/2},
	\]
	since $\mathcal L_\nu$ is a nonnegative self-adjoint in $L^2(\mathbb R^n_+)$.
	
	Similarly to the proofs of \eqref{eq1-proof of mainthm} and \eqref{eq2-proof of mainthm}, we can show that $S^j_{\mathcal L_\nu}$ also satisfies the estimates \eqref{eq1-proof of mainthm} and \eqref{eq2-proof of mainthm} with $q_0=2$ and any $\f{1}{1-\gamma}<p_0<2$,  and $S^j_{\mathcal L_\nu}$ taking place of $R_\nu$. Hence, by Theorem \ref{BZ-thm},  $S^j_{\mathcal L_\nu}$ is bounded on  $L^p(\mathbb R^n_+)$ for $\f{1}{1-\gamma_\nu}<p\le 2$.
	
	To complete our proof, we will prove that $S^j_{\mathcal L_\nu}$ is bounded on $L^p(\mathbb R^n_+)$ for $2<p<\f{1}{\gamma_{\nu+e_j}}$. To do this, in the light of Theorem \ref{Martell-thm}, we need to prove the following inequalities  with $2<q_0<\f{1}{\gamma_{\nu+e_j}}$ and $m\in \mathbb N$ such that $2m>n+1$,
	 	\begin{eqnarray}\label{e1-Martell SL}
	 	\Big( \fint_{B} \big| S^j_{\mathcal L_\nu}(I-e^{-r_B^2\mathcal L_\nu})^mf\big|^{2}dx\Big)^{1/2} \leq
	 	C \mathcal{M}_{2}(f)(x),
	 \end{eqnarray}
	 and
	 \begin{eqnarray}\label{e2-Martell SL}
	 	\Big( \fint_{B} \big| S^j_{\mathcal L_\nu}[I-(I-e^{-r_B^2\mathcal L_\nu})^m]f\big|^{q_0}dx\Big)^{1/q_0} \leq
	 	C \mathcal{M}_{2}(|Tf|)(x),
	 \end{eqnarray}
	For \eqref{e1-Martell SL}, we write
	\[
	\begin{aligned}
		\Big( \fint_{B} \big| S^j_{\mathcal L_\nu}(I-e^{-r_B^2\mathcal L_\nu})^mf\big|^{2}dx\Big)^{1/2}\le \sum_{i=0}^\vc \Big( \fint_{B} \big| S^j_{\mathcal L_\nu}(I-e^{-r_B^2\mathcal L_\nu})^m(f\chi_{S_i(B)})\big|^{2}dx\Big)^{1/2}.
	\end{aligned}
	\]
 	For $i=0,1,2$, by the $L^2$-boundedness of $S^j_{\mathcal L_\nu}$ and $e^{-t\mathcal L_\nu}$, we have
 	\[
 	\begin{aligned}
 		\Big( \fint_{B} \big| S^j_{\mathcal L_\nu}(I-e^{-r_B^2\mathcal L_\nu})^m(f\chi_{S_i(B)})\big|^{2}dx\Big)^{1/2}&\lesi \Big(\f{1}{|B|}\int_{S_i(B)} |f|^2 dx\Big)^{1/2}\\
 		&\lesi \mathcal M_2 f(x).
 	\end{aligned}
  	\]
	For $j\ge 3$,  similarly to \eqref{eq1-proof of mainthm}, we arrive at
	\[
	\begin{aligned}
		\Big( \fint_{B} \big| S^j_{\mathcal L_\nu}(I-e^{-r_B^2\mathcal L_\nu})^m(f\chi_{S_i(B)})\big|^{2}dx\Big)^{1/2}&\lesi 2^{-i(2m-1-n/2)}\Big(\fint_{S_i(B)} |f|^2 dx\Big)^{1/2}\\
		&\lesi 2^{-i(2m-1-n/2)}\mathcal M_2 f(x).
	\end{aligned}
	\]
	It follows \eqref{e1-Martell SL}.

	It remains to prove \eqref{e2-Martell SL}. Indeed, we have
		$$
	\begin{aligned}
		\Big(\int_B &|S^j_{\mathcal L_\nu}[I-(I-e^{-r_B^2\mathcal L_\nu})^m]f(x)|^{q_0}dx\Big)^{1/q_0}\\
		&\lesi \sum_{k=1}^m \Big(\int_B |S^j_{\mathcal L_\nu}e^{-kr_B^2\mathcal L_\nu}f(x)|^{q_0}dx\Big)^{1/q_0}\\
		&\lesi \sup_{1\leq k\leq m}\left[\int_B \left( \int_0^\vc |(\sqrt t \delta_j)e^{-kr_B^2\mathcal L_\nu}e^{-t\mathcal L_\nu}f(x)|^2\f{dt}{t} \right)^{q_0/2}dx\right]^{1/q_0}.
	\end{aligned}
	$$
	On the other hand, note that for $f\in L^2(\mathbb R^n_+)$ and $\nu \in (-1,\vc)^n$ we have, for $t>0$,
	\begin{equation*}
		e^{-t\mathcal L_{\nu} }f =\sum_{k\in \mathbb N^n} e^{-t(4|k|+2|\nu|+2n)}\langle f,\varphi^\nu_k\rangle \varphi^\nu_k.
	\end{equation*}
	Using this and \eqref{eq- delta and eigenvector}, by a simple calculation we come up with
	\[
	e^{-t(\mathcal L_{\nu + e_j}+\alpha_j)}\delta_j f  = \delta_j e^{-t \mathcal L_\nu} f,
	\]
	where $\alpha_j = 2$ if $\nu_j\ge 0$ and $\alpha_j=2-4\nu_j$ if $\nu_j<0$.
	
	It follows that 
	\[
	(\sqrt t \delta_j)e^{-kr_B^2\mathcal L_\nu}e^{-t\mathcal L_\nu}f= e^{-kr_B^2(\mathcal L_{\nu+e_j}+\alpha_j)}(\sqrt t \delta_j)e^{-t\mathcal L_\nu}f.
	\]
	Consequently,
	$$
	\begin{aligned}
		\Big(\int_B &|S^j_{\mathcal L_\nu}[I-(I-e^{-r_B^2\mathcal L_\nu})^m]f(x)|^{q_0}dx\Big)^{1/q_0}\\
		&\lesi \sup_{1\leq k\leq m}\left[\int_B \left( \int_0^\vc |e^{-kr_B^2(\mathcal L_{\nu+e_j}+\alpha_j)}(\sqrt t \delta_j)e^{-t\mathcal L_\nu}f(x)|^2\f{dt}{t} \right)^{q_0/2}dx\right]^{1/q_0}\\
		&\lesi\sum_{i\geq 0}\sup_{1\leq k\leq m}\left[\int_B \left( \int_0^\vc \left|e^{-kr_B^2(\mathcal L_{\nu+e_j}+\alpha_j)}[(\sqrt t \delta_j)e^{-t\mathcal L_\nu}f\chi_{S_i(B)}](x)\right|^2\f{dt}{t} \right)^{q_0/2}dx\right]^{1/q_0}
	\end{aligned}
	$$
	which along with Minkowski's inequality, Proposition \ref{prop-  heat kernel n ge 1} and Lemma \ref{lem Lpq} gives
	$$
	\begin{aligned}
		\Big(\int_B &|S^j_{\mathcal L_\nu}[I-(I-e^{-r_B^2\mathcal L_\nu})^m]f(x)|^{q_0}dx\Big)^{1/q_0}\\
		&\lesi\sum_{i\geq 0}\sup_{1\leq k\leq m}\left( \int_0^\vc \left\|e^{-kr_B^2(\mathcal L_{\nu+e_j}+\alpha_j)}[(\sqrt t \delta_j)e^{-t\mathcal L_\nu}f\chi_{S_i(B)}]\right\|_{L^{q_0}(B)}^2\f{dt}{t} \right)^{1/2}\\
		&\lesi\sum_{i\geq 0} e^{-c4^i}|B|^{-(\f{1}{2}-\f{1}{q_0})}\left( \int_0^\vc \left\|(\sqrt t \delta_j)e^{-t\mathcal L_\nu}f\right\|_{L^2(S_i(B))}^2\f{dt}{t} \right)^{1/2}\\
		&\lesi\sum_{j\geq 0} e^{-c4^i}|B|^{-(\f{1}{2}-\f{1}{q_0})}\Big(\int_{2^iB} |S_{\mathcal L_\nu}f(x)|^{2}dx\Big)^{1/2}.
	\end{aligned}
	$$
	This implies \eqref{e2-Martell SL}. 
	
	We have proved that the square function  $S^j_{\mathcal L_\nu}$ is bounded on $L^p(\mathbb R^n_+)$ for $\f{1}{1-\gamma_\nu}<p<\f{1}{\gamma_{\nu+e_j}}$.
	
	For the weighted estimate, we fix  $\f{1}{1-\gamma_\nu}<p<\f{1}{\gamma_{\nu+e_j}}$ and $w\in A_{p(1-\gamma_\nu)}\cap RH_{(\f{1}{p\gamma_{\nu+e_j}})'}$. By Lemma \ref{weightedlemma1}, we can there exists $\f{1}{1-\gamma_\nu}<p_0<p<q_0<\f{1}{\gamma_{\nu+e_j}}$ such that $w\in A_{p/p_0}\cap RH_{(\f{q_0}{p})'}$.
	
	Therefore, by Theorem \ref{Martell-thm}, it suffices to verify the following estimates
	we need to prove that
		 	\begin{eqnarray}\label{e1-Martell SL s}
		\Big( \fint_{B} \big| S^j_{\mathcal L_\nu}(I-e^{-r_B^2\mathcal L_\nu})^mf\big|^{p_0}dx\Big)^{1/p_0} \leq
		C \mathcal{M}_{p_0}(f)(x)
	\end{eqnarray}
	and
	 \begin{eqnarray}\label{e2-Martell SL s}
		\Big( \fint_{B} \big| S^j_{\mathcal L_\nu}[I-(I-e^{-r_B^2\mathcal L_\nu})^m]f\big|^{q_0}dx\Big)^{1/q_0} \leq
		C \mathcal{M}_{p_0}(|Tf|)(x).
	\end{eqnarray}
The proof of \eqref{e2-Martell SL s} is similar to \eqref{e2-Martell SL} (see also Theorem 7.2 in \cite{AM2}) by using Lemma 7.4 in \cite{AM2}, Proposition \ref{prop-  heat kernel n ge 1} and the $L^{p_0}-L^{p_0}$-off diagonal estimates instead of $L^{2}-L^{q_0}$-off diagonal estimates in Lemma \ref{lem Lpq}.

It remains to prove \eqref{e1-Martell SL s}. To do this, we write
	\[
	\begin{aligned}
		\Big( \fint_{B} \big| S^j_{\mathcal L_\nu}(I-e^{-r_B^2\mathcal L_\nu})^mf\big|^{p_0}dx\Big)^{1/p_0}\le \sum_{i=0}^\vc \Big( \fint_{B} \big| S^j_{\mathcal L_\nu}(I-e^{-r_B^2\mathcal L_\nu})^m(f\chi_{S_i(B)})\big|^{p_0}dx\Big)^{1/p_0}.
	\end{aligned}
	\]
	For $i=0,1,2$, by the $L^{p_0}$-boundedness of $S^j_{\mathcal L_\nu}$ and $e^{-t\mathcal L_\nu}$, we have
	\[
	\begin{aligned}
		\Big( \fint_{B} \big| S^j_{\mathcal L_\nu}(I-e^{-r_B^2\mathcal L_\nu})^m(f\chi_{S_i(B)})\big|^{p_0}dx\Big)^{1/p_0}&\lesi \Big(\f{1}{|B|}\int_{S_i(B)} |f|^{p_0} dx\Big)^{1/p_0}\\
		&\lesi \mathcal M_{p_0} f(x).
	\end{aligned}
	\]
	For $j\ge 3$, by the Holder's inequality we have
	\[
		\Big( \fint_{B} \big| S^j_{\mathcal L_\nu}(I-e^{-r_B^2\mathcal L_\nu})^m(f\chi_{S_i(B)})\big|^{p_0}dx\Big)^{1/p_0}\le 	\Big( \fint_{B} \big| S^j_{\mathcal L_\nu}(I-e^{-r_B^2\mathcal L_\nu})^m(f\chi_{S_i(B)})\big|^{2}dx\Big)^{1/2}.
	\]
	At this stage, similarly to \eqref{eq1-proof of mainthm} by using the $L^{p_0}-L^2$-off diagonal estimates instead of $L^2-L^2$-off diagonal estimate we obtain  \eqref{e1-Martell SL s}.

	Hence the proof is complete.
\end{proof}

{\bf Acknowledgement.} The author was supported by the research grant ARC DP140100649 from the Australian Research Council. The author would like to thank Jorge Betancor for useful discussion.


\begin{thebibliography}{999}

\bibitem{AM} P. Auscher and J.M. Martell, Weighted norm inequalities, off-diagonal estimates and elliptic
operators. Part I: General operator theory and weights,
Adv. in Math. 212 (2007), 225--276.

\bibitem{AM2} P. Auscher and J.M. Martell, Weighted norm inequalities, off-diagonal estimates and elliptic operators. Part III: Harmonic analysis of elliptic operators,  J. Funct. Anal. 241 (2006), 703--746.

\bibitem{BZ} F. Bernicot and J. Zhao, New abstract Hardy spaces, J. Funct.
Anal. 255 (2008), 1761--1796.



 
\bibitem{Betancor} J. J. Betancor, J. C. Farina, L. Rodr\'iguez-Mesa and A. Sanabria-Garc\'ia, Higher order Riesz transforms for Laguerre expansions, Illinois Journal of Mathematics 55 (2011), 27--68.

\bibitem{Betancor2} J. Betancor, J.C. Fariña, L. Rodr\'iguez-Mesa, A. Sanabria and J.L. Torrea, Transference between Laguerre and Hermite settings, J. Funct. Anal.  254 (2008), 826--850.

\bibitem{Betancor3} J. J. Betancor, S. M. Molina and  L. Rodr\'iguez-Mesa, Area Littlewood–Paley functions associated with Hermite and Laguerre Operators, Potential Anal. 34 (2011), 345--369.

\bibitem{B} T. A. Bui, Riesz transforms, Hardy spaces and Campanato spaces associated with Laguerre expansions. Available at: https://www.researchgate.net/publication/378302780\_Riesz\_transforms\_Hardy\_spaces\_and\_Campanato\_spaces\_associated\_with\_Laguerre\_expansions




\bibitem{CD}  T. Coulhon and X. T. Duong, Maximal regularity and kernel bounds: observations on a theorem by Hieber and Pr\"uss, Adv. Differential Equations 5 (2000), no. 1--3, 343--368.

\bibitem{Du} J. Duoandikoetxea, Fourier Analysis, Grad. Stud. math, 29, American Math. Soc., Providence,
2000.





\bibitem{Dziu} J. Dziuba\'nski, Hardy Spaces for Laguerre Expansions, Constr. Approx., 27 (2008), 269--287


\bibitem{FS} C. Fefferman and E.M. Stein, $H^p$ spaces of several variables, Acta Math. 129 (1972), 137--193.


 


\bibitem{JN} R. Johnson and C.J. Neugebauer, Change of variable results for $A_p$-and reverse H\"older
$RH_r$-classes, Trans. Amer. Math. Soc. 328 (1991), no. 2, 639--666.
 

\bibitem{L} N. N. Lebedev, \emph{Special Functions and Their applications}, Dover, New York, 1972.

\bibitem{Muc} B. Muckenhoupt, Conjugate functions for Laguerre expansions, Trans. Amer. Math. Soc. 147 (1970), 403--418.


\bibitem{MS} B. Muckenhoupt, E.M. Stein, Classical expansions and their relation to conjugate harmonic functions, Trans. Amer. Math. Soc. 118 (1965), 17--92.

\bibitem{Na} I. N\"asell, Rational bounds for ratios of modified Bessel functions, SIAM J. Math. Anal. 9 (1978), 1--11.

\bibitem{NS} A. Nowak and K. Stempak, Riesz transforms and conjugacy for Laguerre function expansions of Hermite type, J.  Funct. Anal.  244 (2007),  399--443.

\bibitem{NS2} A. Nowak and K. Stempak,  Adam Nowak,  Riesz transforms for multi-dimensional Laguerre function expansions, Adv. in Math.   215 (2007), 642--678,

 


\bibitem{S} E. Stein, {\it Harmonic Analysis: real-variable methods, orthogonality, and oscillatory integrals}, Princeton Univ. Press, 1993.

\bibitem{ST}   K. Stempak and J.L. Torrea, Poisson integrals and Riesz transforms for Hermite function expansions with weights, J. Funct. Anal. 202 (2003), 443--472.

\bibitem{Th} S. Thangavelu, Lectures on Hermite and Laguerre Expansions, Math. Notes, vol. 42, Princeton Univ. Press, Princeton, NJ, 1993.








\end{thebibliography}
\end{document}